\documentclass[reqno,11pt]{amsart}
\usepackage{amsmath,amssymb,latexsym,soul,cite,mathrsfs,accents}
\usepackage[dvipsnames]{xcolor}
\usepackage{color,enumitem,graphicx}
\newif\ifdraft
\draftfalse

\definecolor{JHAcolor}{rgb}{0.85,0.45,0.0}
\definecolor{TFcolor}{rgb}{0.0,0.55,0.0}
\definecolor{PPcolor}{rgb}{0.55,0.0,0.55}
\definecolor{MYcolor}{rgb}{0.0,0.45,0.70}
\ifdraft
  \newcommand{\JHA}[1]{\leavevmode\pdfannot width 10bp height 10bp depth 0bp {/Subtype /Text /Open false /Name /Note /C [0 0 1] /T (JHA) /Contents (#1)}\ignorespaces}
  \newcommand{\JHAblock}[1]{\leavevmode\pdfannot width 10bp height 10bp depth 0bp {/Subtype /Text /Open false /Name /Note /C [0 0 1] /T (JHA) /Contents (#1)}\ignorespaces}
  \newcommand{\TF}[1]{\leavevmode\pdfannot width 10bp height 10bp depth 0bp {/Subtype /Text /Open false /Name /Note /C [1 0 0] /T (TF) /Contents (#1)}\ignorespaces}
  \newcommand{\TFblock}[1]{\leavevmode\pdfannot width 10bp height 10bp depth 0bp {/Subtype /Text /Open false /Name /Note /C [1 0 0] /T (TF) /Contents (#1)}\ignorespaces}

\else
  \newcommand{\JHA}[1]{}
  \newcommand{\JHAblock}[1]{}
  \newcommand{\TF}[1]{}
  \newcommand{\TFblock}[1]{}

\fi
\usepackage[colorlinks=true,urlcolor=blue,
citecolor=red,linkcolor=blue,linktocpage,pdfpagelabels,
bookmarksnumbered,bookmarksopen]{hyperref}
\usepackage{marginnote}
\usepackage[english]{babel}
\usepackage[left=2.6cm,right=2.6cm,top=2.9cm,bottom=2.9cm]{geometry}
\usepackage{tensor}
\newtheorem{theorem}{Theorem}[section]
\newtheorem{lemma}[theorem]{Lemma}

\newtheorem{proposition}[theorem]{Proposition}
\newtheorem{remark}[theorem]{Remark}
\newtheorem{definition}[theorem]{Definition}

\newtheorem{lemmaletter}{Lemma}

\numberwithin{equation}{section}

\newcommand{\innerthmname}{}

\theoremstyle{definition}

\makeatletter
\def\namedlabel#1#2{\begingroup
	#2%
	\def\@currentlabel{#2}%
	\phantomsection\label{#1}\endgroup
}

\def\XXint#1#2#3{{\setbox0=\hbox{$#1{#2#3}{\int}$ }
		\vcenter{\hbox{$#2#3$ }}\kern-.6\wd0}}

\newcommand*\owedge{\mathpalette\@owedge\relax}
\newcommand*\@owedge[1]{%
	\mathbin{%
		\ooalign{%
			$#1\m@th\bigcirc$\cr
			\hidewidth$#1\m@th\wedge$\hidewidth\cr
		}%
	}%
}
\makeatother

\newcommand{\ud}{\mathrm{d}}
\newcommand{\loc}{\mathrm{loc}}

\allowdisplaybreaks[3]

\title[Delaunay solutions to critical Hartree equations]{Delaunay solutions to the fractional Hartree equation with critical growth}
\thanks{This work was partially supported by Funda\c c\~ao de Amparo \`a Pesquisa do Estado de S\~ao Paulo (FAPESP), Conselho Nacional de Desenvolvimento Cient\'ifico e Tecnol\'ogico (CNPq), National Science Foundation of China (NSFC), and Natural Science Foundation of Zhejiang Province (ZJNSF).}

\author[J.H. Andrade]{Jo\~{a}o Henrique Andrade}
\address[J.H. Andrade]{Institute of Mathematics and Statistics,
	University of S\~ao Paulo
	\newline\indent
	05508-090, S\~ao Paulo-SP, Brazil}
\email{\href{mailto:andradejh@ime.usp.br}{andradejh@ime.usp.br}}

\author[T. Feng]{Tao Feng}
\address[T. Feng]{School of Mathematical Sciences,
	Zhejiang Normal University
	\newline\indent
	321004, Jinhua-ZJ, People's Republic of China}
\email{\href{mailto:fengtao@zjnu.edu.cn}{fengtao@zjnu.edu.cn}}

\author[P. Piccione]{Paolo Piccione}
\address[P. Piccione]{ Department of Mathematics, 
    School of Sciences, Great Bay University
    \newline\indent 
    523000, Dongguan-GD, People’s Republic of China
	\newline\indent
	and
	\newline\indent
	School of Mathematical Sciences,
	Zhejiang Normal University
	\newline\indent
	321004, Jinhua-ZJ, People's Republic of China
    \newline\indent
    \phantom{and}
    \newline\indent
    (permanent address) Institute of Mathematics and Statistics,
	University of S\~ao Paulo
	\newline\indent
	05508-090, S\~ao Paulo-SP, Brazil}
\email{\href{mailto:piccione@ime.usp.br}{paolo.piccione@usp.br, paolo.piccione@gbu.edu.cn}}

\author[M. Yang]{Minbo Yang}
\address[M. Yang]{School of Mathematical Sciences,
	Zhejiang Normal University
	\newline\indent
	321004, Jinhua-ZJ, People's Republic of China}
\email{\href{mailto:mbyang@zjnu.edu.cn}{mbyang@zjnu.edu.cn}}

\date{}

\begin{document}
	
	\begin{abstract}
We study positive solutions of the critical fractional Hartree equation with a non-removable isolated singularity at the origin. This equation is doubly nonlocal, involving both the fractional Laplacian and a Riesz convolution potential. We first prove that every positive singular solution is radially symmetric about the origin, by combining the Caffarelli--Silvestre extension with the method of moving spheres. We then establish the existence of Delaunay-type periodic singular solutions. After the Emden--Fowler transformation, the Hartree convolution survives as a genuinely nonlocal integral term, so that the resulting periodic equation cannot be reduced to an ordinary differential equation. We construct nonconstant periodic solutions for all sufficiently large periods by minimizing a Rayleigh-type quotient in a periodic fractional Sobolev space.
\end{abstract}
	
	\maketitle

\section{Introduction}

\noindent{\bf Motivation.}
Singular solutions to nonlocal PDEs with critical exponents arise in geometry and dispersive models
\cite{CaffarelliSilvestre2007,FrankLenzmannSilvestre2016,Lions1985a,Lions1985b}. In this paper, we study the fractional Hartree equation
\begin{equation}\label{eq:frac Hartree}\tag{$\mathcal{P}_{n,s,\alpha}$}
	(-\Delta)^s u = \big(\mathcal{R}_\alpha * u^{p_{\ast}}\big)u^{p_{\ast}-1}
	\quad {\rm in} \quad \mathbb{R}^n\setminus\{0\},
\end{equation}
where $\mathcal{R}_\alpha(x):=|x|^{\alpha-n}$ is the Riesz potential. Here $n\geqslant2$, $s\in(0,1)$, $\alpha\in(0,n)$, and the exponent $p_{\ast}= (n+\alpha)/(n-2s)$ is the Hardy--Littlewood--Sobolev critical exponent \cite{MR1544927,MR165337,MR717827}, for which equation \eqref{eq:frac Hartree} is conformally invariant. Hartree equations arise in quantum mechanics, astrophysics, and nonlinear optics, and constitute one of the most important classes of nonlocal elliptic equations.

Equation \eqref{eq:frac Hartree} admits both regular solutions on the whole space and solutions with an isolated singularity at the origin. Singular solutions govern the local behavior near blow-up points and concentration sites, and their classification is directly linked to the compactness of solution spaces. A qualitative theory of \eqref{eq:frac Hartree} therefore requires understanding the symmetry of singular solutions, their asymptotic behavior near the origin, and the existence of periodic Delaunay-type profiles.

A distinctive feature of equation \eqref{eq:frac Hartree} is that it involves two independent nonlocal terms, namely the fractional Laplacian and the Riesz convolution. In fact, by using the properties of the Riesz potential,   the classical Yamabe equation can be formally recovered, up to a multiplicative constant, from the critical Hartree equation \eqref{eq:frac Hartree} by letting $\alpha\to0^+$ or $\alpha\to n^-$.
 For the classical Yamabe equation, the Emden--Fowler transformation reduces the problem to an autonomous ODE; in the fractional case, one obtains a single nonlocal equation instead. The Hartree convolution, however, survives the change of variables, so that the transformed equation retains nonlocal integral nonlinearity. In particular, the ODE phase-plane analysis, integral representations, and Caffarelli--Silvestre extension techniques developed for the Yamabe problem all require substantial modification in order to handle the Hartree interaction.

\medskip

\noindent{\bf Setting and notation.}
The fractional Laplacian $(-\Delta)^{s}$ is defined by
\begin{equation}\label{fraction laplace}
	(-\Delta)^{s}u(x):=\kappa_{n,s} \,\mathrm{P.V.} \int_{\mathbb{R}^{n}}\frac{u(x)-u(y+x)}{|y|^{n+2s}}\,\ud y
\end{equation}
with
\begin{equation*}
	\kappa_{n,s}=\pi^{-\frac{n}{2}}2^{2s}\frac{\Gamma(\frac{n+2s}{2})}{\Gamma(1-s)}s.
\end{equation*}
We consider positive singular solutions $u\in \mathcal{C}^{1,1}_{\loc}(\mathbb{R}^{n}\setminus\{0\}) \cap L_{s}(\mathbb{R}^{n}) \cap L^{1}_{\loc} (\mathbb{R}^{n})$ to equation \eqref{eq:frac Hartree}, where
\begin{equation}\label{Ls}
	L_{s}(\mathbb{R}^{n}):=\left\{u: \int_{\mathbb{R}^{n}}\frac{|u(y)|}{1+|y|^{n+2s}}  \,\ud y <+\infty\right\}.
\end{equation}
To study their qualitative properties, we employ the Caffarelli--Silvestre extension. Let
$X=(x,t)\in \mathbb{R}^{n+1}_{+}:=\mathbb{R}^{n}\times\mathbb{R}_{+}$.
In \S\ref{Extension formulation}, we construct the Caffarelli--Silvestre extension \(U\) of \(u\) by \eqref{U(x,t)} and prove that it satisfies the extension problem
\begin{equation}\label{eq:extension of u}
	\left\{
	\begin{aligned}
		&-\operatorname{div}\!\left(t^{1-2s}\nabla U\right)=0
		&&{\rm in}\ \mathbb{R}^{n+1}_{+},\\
		&\frac{\partial U}{\partial\nu^s}(x,0)=
		\left(\mathcal{R}_\alpha\ast u^{p_{*}}\right)u^{p_{*}-1}
		&&{\rm on}\ \mathbb{R}^{n}\setminus\{0\},
	\end{aligned}
	\right.
\end{equation}
where $\frac{\partial U}{\partial\nu^s}(x,0)=-\lim_{t\to0^{+}} t^{1-2s}\partial_t U(x,t)$ and $u(x)=U(x,0)$. Hence, instead of dealing directly with the nonlocal operator \((-\Delta)^s\), we study the local degenerate elliptic equation \eqref{eq:extension of u} in one higher dimension.

\medskip

\noindent{\bf Main results.}
Our first main result establishes the radial symmetry of positive solutions with non-removable isolated singularities.

\begin{theorem}[Radial Symmetry]\label{symmetric}
	Let $n \geqslant 2$, $s \in (0,1)$, and $\alpha \in (0,n)$. Suppose $u\in \mathcal{C}^{1,1}_{\loc}(\mathbb{R}^{n}\setminus\{0\}) \cap L_{s}(\mathbb{R}^{n}) \cap L^{1}_{\loc} (\mathbb{R}^{n})$ is a positive solution to \eqref{eq:frac Hartree} with a non-removable singularity at the origin. Then $u$ is radially symmetric with respect to the origin.
\end{theorem}

The radial symmetry established in Theorem~\ref{symmetric} reduces the study of isolated singular solutions of \eqref{eq:frac Hartree} to the radial setting. In the classical Yamabe problem and its fractional counterpart, a distinguished family of radial singular solutions is given by the Delaunay profiles, which are central to the analysis of isolated singularities. Motivated by these developments, it is natural to ask whether equation \eqref{eq:frac Hartree} also admits analogous periodic singular solutions.

We establish the existence of Delaunay-type solutions to \eqref{eq:frac Hartree}, namely solutions of the form
\begin{equation}\label{scaling}
	u(r)=r^{-\frac{n-2s}{2}}v(r)  \quad {\rm on}\ \mathbb{R}^{n}\setminus\{0\}.
\end{equation}
After the Emden--Fowler change of variable $r=e^{t}$, we regard the function as depending on $t$ and still denote it by $v(t)$. Under this change of variables, equation \eqref{eq:frac Hartree} can be written as
\begin{equation}\label{eq:cylinder-Hartree}
	\mathcal{L}_{s}v(t)=\left(\int_{-\infty}^{+\infty}\mathcal{K}_{-\frac{\alpha}{2}}(t-\tau)v(\tau)^{p_{\ast}} \ud\tau\right)v(t)^{p_{\ast}-1},
\end{equation}
where $\mathcal{L}_{s}$ is the linear operator defined by
\begin{equation}\label{Lgamma}
	\mathcal{L}_{s} v(t)
	=
	\kappa_{n,s}\,\mathrm{P.V.}
	\int_{-\infty}^{+\infty}
	\bigl(v(t)-v(\tau)\bigr)
	\mathcal{K}_{s}(t-\tau) \ud\tau
	+
	c_{n,s} v(t),
\end{equation}
$c_{n,s}>0$ depends only on $n$ and $s$, and the kernel
$\mathcal{K}_m$, for $-n/2<m<1$, is given explicitly in \eqref{Km}.
In particular, the kernels $\mathcal{K}_{-\frac{\alpha}{2}}$ in \eqref{eq:cylinder-Hartree} and $\mathcal{K}_{s}$ in \eqref{Lgamma} correspond to the choices $m=-\alpha/2$ and $m=s$, respectively.

If we restrict ourselves to $L$-periodic functions, {\it i.e.},
$v(t)=v(t+L)$, then the operator $\mathcal{L}_{s}$ can be
rewritten as
\begin{equation}\label{eq:Lgamma-periodic}
	\mathcal{L}_{s}^{L} v(t)
	=
	\kappa_{n,s}\,\mathrm{P.V.}
	\int_{0}^{L}
	\bigl(v(t)-v(\tau)\bigr)
	\mathcal{K}_{s}^{L}(t-\tau)\,\ud\tau
	+
	c_{n,s} v(t),
\end{equation}
where $\mathcal{K}_{s}^{L}$ is the associated periodic kernel defined in \eqref{Kml} with $m=s$. In this case,
equation \eqref{eq:cylinder-Hartree} becomes
\begin{equation}\label{eq:cylinder-Hartree-periodic}
	\mathcal{L}_{s}^{L} v(t)
	=
	\left(
	\int_{0}^{L}
	\mathcal{K}_{-\frac{\alpha}{2}}^{L}(t-\tau)
	\, v(\tau)^{p_{\ast}}\, \ud\tau
	\right)
	v(t)^{p_{\ast}-1},
\end{equation}
where $\mathcal{K}_{-\frac{\alpha}{2}}^{L}$ denotes the periodic kernel
corresponding to $m=-\alpha/2$ in \eqref{Kml}. For $L$-periodic solutions, problem \eqref{eq:cylinder-Hartree} is equivalent to finding a minimizer of the functional
\begin{equation}
	\mathscr{F}_{L}(v)=\frac{\frac{\kappa_{n,s}}{2}\int_{0}^{L}\int_{0}^{L}\left(v(t)-v(\tau)\right)^{2}\mathcal{K}_{s}^{L}(t-\tau)\,\ud \tau\,\ud t+c_{n, s} \int_{0}^{L}v(t)^{2}\,\ud t}{\left(\int_{0}^{L}\int_{0}^{L} \mathcal{K}_{-\frac{\alpha}{2}}^{L}(t-\tau)v(\tau)^{p_{\ast}} v(t)^{p_{\ast}}\,\ud \tau\,\ud t\right)^{\frac{1}{p_{\ast}}}}.
\end{equation}
A positive minimizer exists (see Lemma~\ref{lem:existence-minimizer} below), and we denote the minimum value of $\mathscr{F}_L$ by $c(L)$. 

We need to introduce the lower critical weight given by 
\begin{equation}\label{eq:alpha-star}
\alpha_{\ast} = \alpha_{\ast}(n,s) := \max\left\{0,\frac{n(1-6s)+8s^2}{2n-1-2s}\right\}.
\end{equation}
With this definition in hand, our second main result is the following.
\begin{theorem}[Existence of Delaunay solutions]\label{thm:main}
	Let $n \geqslant 2$, $s \in (0,1)$, and $\alpha \in (\alpha_{\ast},n)$. There exists $T_0<\infty$ such that for every $T\geqslant T_0$, equation \eqref{eq:frac Hartree}
	admits a positive, $T$-periodic solution of the form
	\[
	u(x)=|x|^{-\frac{n-2s}{2}}v(\ln|x|),
	\qquad v(t+T)=v(t)>0.
	\]
	These solutions are nonconstant Delaunay profiles for all $T\gg1$ sufficiently large.
\end{theorem}

For $0<s<\frac{1}{2}$, the lower critical weight $\alpha_\ast\geqslant 0$ in \eqref{eq:alpha-star} arises naturally from the $L^\infty$ regularity theory. In the De Giorgi truncation scheme used to establish boundedness of minimizers (Proposition~\ref{L infty for s<1/2}), the periodic kernel $\mathcal{K}_{-\alpha/2}^L$ must be integrable in $L^{q'}$ for an exponent $q'$ satisfying a chain of constraints dictated by the Sobolev embedding dimension ($1-2s$), the nonlinearity exponent ($p_{\ast}$), and the kernel singularity ($\alpha$). The threshold $\alpha_{\ast}$ is precisely the value below which no admissible $q'$ exists, making the iteration break down. On the other hand, no additional lower bound on $\alpha$ is required for $s\geqslant\frac12$; see Remark~\ref{L infty for 0<s<1}. Since $\alpha_\ast=0$ by definition \eqref{eq:alpha-star} when $s\geqslant\frac12$, the condition $\alpha\in(\alpha_\ast,n)$ is equivalent to the standing assumption $0<\alpha<n$. Accordingly, Theorem~\ref{thm:main} is formulated in the unified form $\alpha\in(\alpha_\ast,n)$.

\medskip

\noindent{\bf Ideas of the proofs.}
We briefly describe the strategy for the proof of our main results.

For the symmetry result (Theorem~\ref{symmetric}), unlike approaches based on integral representations, we investigate the symmetry of solutions through the Caffarelli--Silvestre extension. Although the method of moving planes in integral forms avoids the use of the maximum principle, establishing an equivalent integral representation for the fractional Hartree equation is itself highly nontrivial. In fact, under the relatively weak regularity assumptions ensuring that the fractional Laplacian is well-defined, it is difficult even to guarantee the finiteness of the integrals appearing in the corresponding integral representation. Motivated by this observation, we develop a direct method of moving planes for the extension problem, and we establish the radial symmetry of positive solutions.

Our approach differs from that in \cite{CaffarelliJinSireXiong2014}, which studies the symmetry of positive solutions with non-removable isolated singularities to the fractional Yamabe equation. In their work, the extension function $U\in W^{1,2}_{\loc}(t^{1-2s}, \overline{\mathbb{R}^{n+1}_{+}}\setminus\{0\})$ is assumed to be a weak solution to the extension problem, and the radial symmetry of $U$ with respect to the $x$-variables is established directly. In contrast, we start from the original fractional Hartree equation \eqref{eq:frac Hartree} under the natural regularity assumption $u\in \mathcal{C}^{1,1}_{\loc}(\mathbb{R}^{n}\setminus\{0\}) \cap L_{s}(\mathbb{R}^{n}) \cap L^{1}_{\loc} (\mathbb{R}^{n})$. Based on this assumption, we first establish, in \S\ref{The original fractional Hartree equation}, a key property of the Hartree convolution term that is used in the subsequent moving spheres argument. We then construct the Caffarelli--Silvestre extension $U$ associated with $u$, and prove in \S\ref{Extension formulation} that $U\in W^{1,2}_{\loc}(t^{1-2s},\overline{\mathbb{R}^{n+1}_+}\setminus\{0\})$ and satisfies the corresponding extension problem \eqref{eq:extension of u}. With this preparation, we apply the method of moving spheres to the extension problem to establish the radial symmetry of $U$ with respect to the $x$-variables, which in turn yields the radial symmetry of the trace $u=U(\cdot,0)$. All notation related to the extension problem is introduced in \S\ref{Extension formulation}.

The proof of Theorem~\ref{thm:main} follows the variational strategy introduced by DelaTorre--Del Pino--Gonz\'alez--Wei~\cite{MR3694655} for the fractional Yamabe problem. After the Emden--Fowler transformation, equation \eqref{eq:frac Hartree} is transformed into a one-dimensional nonlocal integro-differential equation involving both the fractional operator $\mathcal L_s$ and the Hartree convolution term, rather than a classical local ordinary differential equation. Consequently, the classical ODE approach is no longer applicable, and we formulate the problem as the minimization of the Rayleigh-type quotient $\mathscr{F}_L$ over the periodic Sobolev space $H_L^s$. By the direct method in the calculus of variations together with the compact embedding $H_L^s \hookrightarrow L^q(0,L)$, we obtain the existence of a minimizer. Compared with the fractional Yamabe problem, the presence of the Hartree convolution gives rise to a nonlocal interaction energy in the Rayleigh-type quotient, making both the variational analysis and the regularity theory substantially more involved. In particular, unlike local nonlinearities, the regularity of the Hartree convolution term does not follow directly from that of the solution. To overcome this difficulty, we combine the extension formulation with a De Giorgi truncation argument to prove that every minimizer is a smooth positive solution to the transformed equation. To the best of our knowledge, this is the first use of De Giorgi truncation techniques for Hartree-type equations; moreover, this approach yields a strictly larger admissible parameter range than the classical Moser iteration method. An energy comparison then shows that the energy of the constant solution grows like $L^{1-1/p_*}$ as $L\to\infty$, whereas a fixed compactly supported test function has energy $\mathcal{O}(1)$. Consequently, the minimizer is necessarily nonconstant for all sufficiently large periods, yielding Delaunay-type solutions.

We remark that the assumption
\begin{equation*}
\alpha_{\ast}<\alpha<n \quad\text{ for }\quad 0<s<1/2
\end{equation*}
is not required for the variational construction of minimizers, but arises solely from the technical requirements of the regularity analysis. We note that $\alpha_{\ast}$ is strictly less than $n - 4s$ (the threshold corresponding to $p_{\ast} \geqslant 2$), so that the De Giorgi approach enlarges the admissible range of $\alpha$ beyond what the Moser iteration method achieves. Moreover, as shown in Remark~\ref{rem:lower-bound-vanishes}, this lower bound becomes nonpositive, and hence the restriction on $\alpha$ disappears, when $s \geqslant \frac{3n - \sqrt{9n^2 - 8n}}{8}$.

\medskip

\noindent{\bf Related results.}
The study of singular solutions to elliptic equations with critical Sobolev exponents has long been a central topic in geometric analysis. Singular solutions characterize the local behavior near singularities, and are central to blow-up analysis, compactness theory, the classification of singularities, and the description of asymptotic profiles. The existence, classification, and asymptotic behavior of such solutions have therefore been studied extensively.

The classical Yamabe equation is the prototype. More precisely, for $n\geqslant 3$, let us consider
\begin{equation}\label{yamabe equation}
	-\Delta u=u^{\frac{n+2}{n-2}} \quad {\rm in} \quad \mathbb{R}^{n} \setminus\{0\}.
\end{equation}
Caffarelli, Gidas and Spruck \cite{Caffarelli-Gidas-Spruck} proved that if the singularity at the origin is removable, then every solution is an Aubin--Talenti bubble. Otherwise, every positive singular solution is radially symmetric. Moreover, by means of the Emden--Fowler transformation and the corresponding autonomous ordinary differential equation, they obtained a complete classification of singular solutions, namely the Fowler (Delaunay-type) solutions.

There are also works on studying similar problems for singular solutions of higher-order conformally invariant Yamabe equations with isolated singularities. In the removable singularity case, the classification of positive entire solutions was established for fourth-order equations by Lin \cite{MR1611691}, generalized to higher-order equations by Wei and Xu \cite{WeiXu1999}, and further extended to conformally invariant integral equations by Chen, Li and Ou \cite{MR2200258} and Li \cite{Li2004}. When $n \geqslant 5$ and the origin is a non-removable singularity, Lin \cite{MR1611691} proved that all singular solutions of the fourth-order equation
\[
(-\Delta)^{2} u=u^{\frac{n+4}{n-4}} \quad {\rm in} \quad \mathbb{R}^{n} \setminus \{0\}
\]
are radially symmetric. Later, Guo, Huang, Wang and Wei \cite{GuoHuangWangWei2020}, and Frank and König \cite{FrankKonig2019} classified all singular solutions by ODE analysis.

Now, let us focus on the following fractional higher-order equation 
\begin{equation}\label{eq:LE_fractional}\tag{$\mathcal{P}_{n,s}$}
  (-\Delta)^{\sigma/2}u=u^{\frac{n+\sigma}{n-\sigma}} \quad {\rm in} \quad \mathbb{R}^{n} \setminus \{0\},  
\end{equation}
where $\sigma \in (0,n)$ is such that $\sigma=m+s$ with $m\in\mathbb{N}_0$ and $s\in(0,1)$.
The symmetry analysis of positive solutions to \eqref{eq:LE_fractional} with non-removable isolated singularities is more difficult due to the lack of a maximum principle and the absence of suitable positivity properties for intermediate derivatives. To overcome these difficulties, Chen, Li and Ou \cite{ChenLiOu2005} reformulated the equation as an equivalent conformally invariant integral equation and proved the radial symmetry of positive singular solutions. In contrast, classifying singular solutions becomes harder. As the order of the operator increases, the associated Emden--Fowler equation becomes a higher-order nonlinear ordinary differential equation, rendering the classical ODE approach used in the fourth-order case ineffective. Jin and Xiong \cite{JinXiong2020} reformulated the equation as an equivalent conformally invariant integral equation
\begin{equation*}
\psi(t)=\int_{-\infty}^{+\infty}
K(t-\tau)\psi(\tau)^{\frac{n+2\sigma}{n-2\sigma}}\,\ud \tau,
\end{equation*}
where
\begin{equation*}
K(t)
=
\frac{1}{2^{\frac{n-2\sigma}{2}}}
\int_{\mathbb{S}^{n-1}}
\frac{\ud \xi}
{\left|\cosh t-\xi\right|^{\frac{n-2\sigma}{2}}}
\quad {\rm and} \quad \psi(t)=|x|^{\frac{n-2\sigma}{2}}u(|x|) \quad {\rm with} \quad  t=\ln|x|,
\end{equation*}
and established the existence of Delaunay-type singular solutions. 

The fractional Yamabe equation is a particular case of \eqref{eq:LE_fractional} with $\sigma=\{\sigma\}=s \in (0,1)$, which is a nonlocal conformally invariant analog of the classical Yamabe equation. For the removable case, by developing the method of moving planes in integral forms, Chen, Li and Ou \cite{MR2200258} classified all positive solutions of the corresponding conformally invariant integral equation. This result was subsequently extended by Li \cite{Li2004} under weaker regularity assumptions. Subsequently, Jin, Li and Xiong \cite{JinLiXiong2014} established a Liouville-type theorem via the Caffarelli--Silvestre extension and the method of moving spheres. In the non-removable singularity case, Caffarelli, Jin, Sire and Xiong \cite{CaffarelliJinSireXiong2014} proved the radial symmetry of positive singular solutions by applying the method of moving planes to the corresponding Caffarelli--Silvestre extension problem. Concerning the existence of Delaunay-type solutions, DelaTorre, del Pino, González and Wei \cite{MR3694655} observed that, after the Emden--Fowler transformation, the fractional Yamabe equation is no longer reduced to a classical autonomous ordinary differential equation. Instead, they worked directly with the transformed nonlocal operator, established a variational framework for the resulting equation, and proved the existence of Delaunay-type singular solutions. Similar questions have also been extensively studied for conformally invariant equations with exponential nonlinearities, including the Liouville equation and higher-order Q-curvature equations. For the interested reader, we refer to \cite{Chen-Li1,Chen-Li2,MR1611691,WeiXu1999,Chang-Yang,Chou-Wan,Yang-Yang,Guo-Liu}, for example.

We now ask whether analogous results hold for the conformally invariant fractional Hartree equation \eqref{eq:frac Hartree}. For  the removable singularities case, it reduces to the classification of positive classical solutions in the whole space. By the method of moving planes in integral forms,  Ma, Shang and Zhang \cite{MR3978520} proved that every positive classical solution in $L_s(\mathbb{R}^{n})\cap \mathcal{C}^{1,1}_{\loc}(\mathbb{R}^n)$ is radially symmetric and monotone decreasing about some point by the direct method of moving planes. For several special choices of the parameters, Dai, Fang and Qin \cite{DaiFangQin2018} and Dai, Huang, Qin, Wang and Fang \cite{DaiHuangQinWangFang2019} established regularity and classification results for positive entire solutions to the critical Hartree equation \eqref{eq:frac Hartree}.

In contrast, the theory of non-removable isolated singularities remains essentially unexplored. To the best of our knowledge, neither the radial symmetry of positive singular solutions nor the existence of Delaunay-type singular solutions of \eqref{eq:frac Hartree} has been established. In this paper, we address these two problems.

When $s=1$ in \eqref{eq:frac Hartree}, both the removable and non-removable isolated singularity cases have been investigated. For the removable singularity case, the radial symmetry of positive entire solutions has been studied in the literature \cite{MR4027015,MR3817173,MR3978520}. For the non-removable singularity case, the radial symmetry of positive solutions was established in \cite{AndradeFengPiccioneYang2025} by the method of moving planes in integral forms under suitable global integrability assumptions. In parallel, similar questions have also been investigated for Hartree-type equations with exponential nonlinearities; see, for example, \cite{Gluck,Guo-Peng,Niu,Feng-Yang-Zhou}.

\medskip

\noindent{\bf Open problems.}
The results of this paper suggest several natural open questions.

A natural question is whether the lower bound $\alpha > \alpha_{\ast}$ appearing in Theorem~\ref{thm:main} for $0 < s < 1/2$ can be removed entirely, so that the existence result holds for all $0 < \alpha < n$ and all $s \in (0,1)$. In this paper, we introduce a De Giorgi truncation approach to the $L^\infty$ regularity estimate for Hartree-type equations. This approach is new in the Hartree setting and already enlarges the admissible parameter range beyond what the classical Moser iteration method yields. However, the integrability condition imposed on the periodic kernel in the De Giorgi iteration still requires a positive lower bound on $\alpha$ when $s$ is small. It remains open whether a refined argument can eliminate this restriction altogether.

Theorem~\ref{symmetric} and Theorem~\ref{thm:main} together establish that radial singular solutions exist for all sufficiently large periods, but a full description of all positive radial singular solutions, analogous to the Caffarelli--Gidas--Spruck classification for the Yamabe equation, remains to be obtained. In a related direction, the variational construction of Theorem~\ref{thm:main} yields nonconstant periodic solutions only for periods $T \geqslant T_0$. Whether nonconstant periodic solutions exist for small periods remains open. This is a genuinely nonlocal phenomenon: in the classical ODE case ($s = 1$, $\alpha = 0$), Fowler solutions exist for all periods, whereas the nonlocal convolution term prevents a direct phase-plane argument. Closely related is the question of uniqueness of periodic solutions for a given period.

Finally, for parameters $(s,\alpha)$ close to the local regime, {\it i.e.} $s\sim 1$ and $\alpha \sim 0$, perturbative methods may extend classification results from the well-understood ODE setting to the nonlocal Hartree equation. We plan to address this problem in a forthcoming work.

\medskip

\noindent{\bf Organization.}
This paper is organized as follows. In \S\ref{Global solutions with an isolated singularity}, we establish the radial symmetry of positive singular solutions to the fractional Hartree equation \eqref{eq:frac Hartree} and prove Theorem~\ref{symmetric}. In \S\ref{sec:prelim}, we reformulate the fractional Hartree equation \eqref{eq:frac Hartree} as a periodic problem via the Emden--Fowler transformation and derive the corresponding extension formulation. In \S\ref{sec:regularity}, we introduce the function spaces required for the variational analysis, prove the compact embedding, and develop the regularity theory for the periodic problem. In \S\ref{sec:minimization}, we prove the existence of minimizers and complete the proof of Theorem~\ref{thm:main}.

\medskip

\noindent{\bf Notation.}
We collect the main notation used throughout the paper.
\smallskip
\begin{itemize}[leftmargin=2em, itemsep=2pt]
\item[--] $n \geqslant 2$ is the spatial dimension;
\item[--] $s \in (0,1)$ is the fractional parameter;
\item[--] $\alpha \in (0,n)$ is the Riesz exponent;
\item[--] $p_{\ast} := (n+\alpha)/(n-2s)$ is the Hardy--Littlewood--Sobolev critical exponent;
\item[--] $\alpha_{\ast} = \alpha_{\ast}(n,s)$ is the lower critical weight defined in \eqref{eq:alpha-star};
\item[--] $\mathcal{R}_\alpha(x) := |x|^{\alpha-n}$ is the Riesz potential;
\item[--] $(-\Delta)^s$ is the fractional Laplacian with normalization constant $\kappa_{n,s}$;
\item[--] $c_{n,s} > 0$ is the zero-order coefficient from the Emden--Fowler transformation;
\item[--] $d_s > 0$ is the Dirichlet-to-Neumann constant in the Caffarelli--Silvestre extension;
\item[--] $L_{s}(\mathbb{R}^{n})$ is the weighted integrability class for $(-\Delta)^s$ in \eqref{Ls};
\item[--] $\mathbb{R}^{n+1}_+$ is the upper half-space where points are denoted $X = (x,t)$;
\item[--] $\mathcal{B}_R^+$ is the upper half-ball of radius $R$; 
\item[--] $B_R$ is the ball of radius $R$ in $\mathbb{R}^n$;
\item[--] $\partial' \mathcal{B}_R^+$ is the flat boundary; 
\item[--] $\partial'' \mathcal{B}_\lambda(X)$ is the spherical boundary;
\item[--] $U$ is the Caffarelli--Silvestre extension of $u$; 
\item[--] $\partial U/\partial \nu^s$ is the conormal derivative;
\item[--] $U_{X,\lambda}$ is the Kelvin transform of $U$ centered at $X$ with radius $\lambda$;
\item[--] $\bar{\lambda}(x)$ is the critical radius for the moving spheres method;
\item[--] $r = e^t$ is the Emden--Fowler change of variables; 
\item[--] $v(t) := r^{(n-2s)/2}\,u(r)$ is the Emden--Fowler transform;
\item[--] $\mathcal{K}_m(t)$ is the kernel in Emden--Fowler coordinates; 
\item[--] $\mathcal{K}_m^L$ is the $L$-periodization of the kernel in Emden--Fowler coordinates;
\item[--] $\mathcal{L}_s$ is the transformed fractional operator acting on the Emden--Fowler variable;
\item[--] $H_L^s$ is the periodic Sobolev space with finite $\mathcal{L}_s$-energy in \S\ref{sec:regularity};
\item[--] $\mathscr{F}_L$ is the Rayleigh quotient; $\mathcal{E}_L$ is the energy functional; 
\item[--] $\mathcal{H}_L$ is the Hartree interaction;
\item[--] $c(L)$ is the ground-state energy; 
\item[--] $v_L$ is the minimizer achieving $c(L)$;
\item[--] $\mathbf{c}_L$ is the constant solution; see \eqref{eq:constant-state};
\item[--] $\mathrm{o}(1)$ denotes a quantity tending to zero in the indicated limit;
\item[--] $\mathcal{O}(\cdot)$ denotes a quantity bounded by a constant times the argument;
\item[--] $\ud x$ is the Lebesgue measure; 
\item[--] $\ud\sigma$ is the surface measure on $\mathbb{S}^{n-1}$.
\end{itemize}

    \section{Global solutions with an isolated singularity}\label{Global solutions with an isolated singularity}

    In this section, we prove the radial symmetry of positive singular solutions to the fractional Hartree equation (Theorem~\ref{symmetric}). We first collect basic properties of positive solutions in the punctured space. We then introduce the Caffarelli--Silvestre extension and prove that the extension satisfies the degenerate elliptic problem \eqref{eq:extension of u}. Finally, we establish the radial symmetry by applying the method of moving spheres to the extension.

    \subsection{The original fractional Hartree equation}\label{The original fractional Hartree equation}
    In this subsection, we first discuss several basic properties of positive solutions to the fractional Hartree equation \eqref{eq:frac Hartree}. 
    
    Throughout this paper, we assume that $u$ is a positive solution to \eqref{eq:frac Hartree} satisfying
    \begin{equation}\label{assumption of u}
        u\in \mathcal{C}^{1,1}_{\loc}(\mathbb R^n\setminus\{0\})\cap L_s(\mathbb R^n) \cap L^{1}_{\loc}(\mathbb{R}^{n}), 
    \end{equation}
    where $L_{s}(\mathbb R^n)$ is defined by \eqref{Ls}. Since $u$ satisfies the regularity assumption in \cite[Proposition~2.4]{Silvestre2007}, the principal value integral in \eqref{fraction laplace} is well-defined. Consequently, $(-\Delta)^s u$ is a continuous function in $\mathbb R^n\setminus\{0\}$.

    Let $K\Subset\mathbb R^n\setminus\{0\} $ be an arbitrary compact set. Since $u\in \mathcal{C}^{1,1}_{\loc}(\mathbb R^n\setminus\{0\})$ is positive, there exist constants $0<c_{K}\leqslant C_{K}<\infty$ such that $c_{K} \leqslant u(x)\leqslant C_K$ for all $x\in K$. Moreover, by the continuity of the fractional Laplacian $(-\Delta)^{s} u$ in $\mathbb R^n\setminus\{0\}$, we have $\sup_{x\in K}|(-\Delta)^s u(x)|<+\infty$. Therefore, by equation \eqref{eq:frac Hartree}, there exists a constant $M_{K}<\infty$ such that
    \begin{equation}\label{boundedness of the convolution term}
        \sup_{x\in K}\big(\mathcal{R}_\alpha * u^{p_{\ast}}\big)(x)\leqslant M_{K}.
    \end{equation}

    \subsection{Extension formulation}\label{Extension formulation}
    To investigate the symmetry properties of solutions to \eqref{eq:frac Hartree}, we employ the Caffarelli--Silvestre extension formulation, which transforms the nonlocal fractional equation into a local degenerate elliptic problem in the upper half-space.

    We first introduce some notation related to the extension problem. We use capital letters to denote points in the upper half-space
    \begin{equation*}
        \mathbb R^{n+1}_+:=\left\{X=(x,t)\in\mathbb R^n\times\mathbb R:t>0\right\}.
    \end{equation*}
    For $R>0$, define $\mathcal{B}_R:=\left\{X\in\mathbb R^{n+1}:|X|<R\right\}$, and $\mathcal B_R^+:=\mathcal B_R\cap\mathbb R^{n+1}_+$. Moreover, let $B_R:=\left\{x\in\mathbb R^n:|x|<R\right\}$, and $\partial'\mathcal B_R^+:=\mathcal B_R\cap\{t=0\}=B_R$. 

   We now consider the Caffarelli--Silvestre extension of $u$ given by 
    \begin{equation}\label{U(x,t)}
        U(x,t) = \int_{\mathbb{R}^{n}} P_t(y-x)\,u(y)\,\ud y,
    \end{equation}
    where
    \begin{equation}
        P_t(z)=p_{n,s}\frac{t^{2s}}
{\bigl(|z|^2+t^2\bigr)^{\frac{n+2s}{2}}} \quad {\rm and} \quad \int_{\mathbb{R}^{n}}P_{t}(z)\,\ud z=1.
    \end{equation}
    The following lemma shows that this extension is well-defined and yields a weak solution to the corresponding extension problem away from the origin.
    \begin{lemma}[Weak solution to the extension problem away from the origin]\label{Weak solution to the extension problem away from the origin}
        Let $n \geqslant 2$, $s \in (0,1)$, and $\alpha \in (0,n)$. Assume that $u \in \mathcal{C}^{1,1}_{\loc}(\mathbb R^n\setminus\{0\})\cap L_s(\mathbb R^n) \cap L^{1}_{\loc}(\mathbb{R}^{n})$ is a positive solution to \eqref{eq:frac Hartree}. Let $U$ be the Caffarelli--Silvestre extension of $u$ defined by \eqref{U(x,t)}. Then $U$ is a positive weak solution to the extension problem \eqref{eq:extension of u}. More precisely, $U\in W^{1,2}_{\loc}(t^{1-2s},\overline{\mathbb R^{n+1}_+}\setminus\{0\})$ and $U$ satisfies the equation \eqref{eq:extension of u} in the distributional sense away from the origin. Here, $U\in W^{1,2}_{loc}\bigl(t^{1-2s},\overline{\mathbb R^{n+1}_+}\setminus\{0\}\bigr)$ means that for every $R>\varepsilon>0$, $U\in W^{1,2}\bigl(t^{1-2s},\mathcal{B}_R^{+}\setminus \overline{\mathcal{B}_\varepsilon^{+}}\bigr)$, that is,
        \begin{equation*}
           \int_{\mathcal{B}_R^{+}\setminus \overline{\mathcal{B}_\varepsilon^{+}}}t^{1-2s}\bigl(|U|^2+|\nabla U|^2\bigr)\,\ud X<\infty . 
        \end{equation*}
    \end{lemma}
    \begin{proof}
    We divide the proof into several claims.
    
    \noindent{\bf Claim~1:} {\it $U$ is well-defined in $\mathbb{R}^{n+1}_+$.}

    \noindent{\it Proof.}
    Let $(x,t)\in \mathbb R^{n+1}_+$ be fixed. Choose $R:=2(|x|+1)$. We split
    \begin{equation*}
        \int_{|y|\leqslant R}P_t(y-x)u(y)\,\ud y
+
\int_{|y|>R}P_t(y-x)u(y)\,\ud y
=I_{1}+I_{2}.
    \end{equation*}

For the local part $I_1$, since $t>0$ is fixed, the kernel $P_t$ is bounded. Hence, 
\begin{equation*}
    |I_1|
\leqslant
C(t)
\int_{|y|\leqslant R}|u(y)|\,\ud y<\infty,
\end{equation*}
where we used the fact that $u\in L^1_{\loc}(\mathbb R^n)$.

For the tail part $I_2$, if $|y|>R$, then
\begin{equation*}
    |y-x|\geqslant |y|-|x| > \frac{|y|}{2}.
\end{equation*}
Therefore,
\begin{equation*}
    |I_2|
\leqslant
C(t)
\int_{|y|>R}
\frac{|u(y)|}{|y|^{n+2s}}
\, \ud y\leqslant 2C(t)\int_{|y|>R}
\frac{|u(y)|}{1+|y|^{n+2s}} \,\ud y<\infty,
\end{equation*}
where we used $u\in L_s(\mathbb R^n)$.

Combining the above estimates, we obtain $|U(x,t)|<\infty$. Hence $U$ is well-defined in $\mathbb R^{n+1}_+$, and Claim~1 is proved.

\medskip
\noindent{\bf Claim~2:} {\it $U$ solves the extension problem \eqref{eq:extension of u} away from the origin.}

\noindent{\it Proof.} Since $u\in \mathcal{C}^{1,1}_{\loc}(\mathbb R^n \setminus \{0\})\cap L_s(\mathbb R^n)\cap L^{1}_{\loc}(\mathbb{R}^{n})$, for every \(x_0\neq0\), there exists \(r>0\) such that
\[
u\in \mathcal{C}^{1,1}(B_r(x_0))\cap L_s(\mathbb R^n) \cap L^{1}_{\loc}(\mathbb{R}^{n}).
\]
Therefore, by Section 1 in \cite{ChenLiMa}, the Caffarelli--Silvestre extension \(U\) satisfies
\begin{equation*}
    -\operatorname{div}(t^{1-2s}\nabla U)=0
\quad {\rm in}\  \mathbb R^{n+1}_+
\end{equation*}
and
\begin{equation*}
    U(x,t)\to u(x)
\quad {\rm as}\ t\to0^+
\end{equation*}
for every $x\neq0$, and Claim~2 is proved.

\medskip
\noindent{\bf Claim~3:} {\it For every $R > \varepsilon > 0$, $U \in W^{1,2}(t^{1-2s}, \mathcal{B}_R^{+} \setminus \overline{\mathcal{B}_\varepsilon^{+}})$.}

\noindent{\it Proof.} We show that for every $R>\varepsilon>0$, 
        \begin{equation*}
        \int_{\mathcal{B}_R^{+}\setminus \overline{\mathcal{B}_\varepsilon^{+}}}t^{1-2s}\bigl(|U|^2+|\nabla U|^2\bigr)\,\ud X<\infty . 
        \end{equation*}
We only prove the finiteness of the gradient term
\begin{equation*}
    \int_{\mathcal{B}_R^{+}\setminus \overline{\mathcal{B}_\varepsilon^{+}}}
t^{1-2s}|\nabla U|^2\,\ud X,
\end{equation*}
since the estimate for the lower-order term follows from a similar and simpler argument. 

Define $f(x,t,y):=P_{t}(y-x)u(y)$. We have 
\begin{equation*}
    \nabla_xf(x,t,y)=\left(\nabla_{x}P_{t}(y-x)\right)u(y) \quad {\rm and} \quad \partial_{t}f(x,t,y)=\left(\partial_{t}P_{t}(y-x)\right)u(y).
\end{equation*}
By the change of variables $z=y-x$, we have 
\begin{equation}\label{partialx P and partialt P}
    \partial_{x_i}P_t(y-x)=-\partial_{z_{i}}P_{t}(z)
=
p_{n,s}(n+2s)
\frac{t^{2s}z_i}
{\left(|z|^2+t^2\right)^{\frac{n+2s+2}{2}}} \quad {\rm and}   \quad \partial_t P_t(z)
=
p_{n,s}
\frac{
t^{2s-1}\bigl(2s|z|^2-nt^2\bigr)
}
{\left(|z|^2+t^2\right)^{\frac{n+2s+2}{2}}}.
\end{equation}
Let $\Omega:=\mathcal{B}_R^+\setminus\overline{\mathcal{B}_\varepsilon^+}$ and $(x,t) \in \Omega$. Define $\Omega_{1}=\{(x,t) \in \Omega:\varepsilon /2\leqslant t<R\}$ and $\Omega_{2}=\{(x,t) \in \Omega:0<t<\varepsilon/2\}$.

We first show that 
\begin{equation}\label{sobolve energy of nablax U}
        \int_{\mathcal{B}_R^{+}\setminus \overline{\mathcal{B}_\varepsilon^{+}}}t^{1-2s}|\nabla_{x}U|^2\, \ud X<\infty . 
        \end{equation}
By \eqref{partialx P and partialt P}, 
\begin{equation}\label{property of nablax P}
    \left|\nabla_{z}P_{t}(z) \right| \leqslant C\frac{t^{2s}|z|}{\left(|z|^{2}+t^{2}\right)^{\frac{n+2s+2}{2}}} \quad {\rm and}  \int_{B_{r}}\nabla_{z}P_{t}(z)\,\ud z=0,
\end{equation}
where $r \in (0,+\infty ]$. 

In $\Omega_{1}$, since $|t| \geqslant \frac{\varepsilon}{2}$, we get
\begin{equation*}
    \left|\nabla_{z}P_{t}(z) \right| \leqslant \frac{C_{\varepsilon}}{\left(1+|z|\right)^{n+2s+1}}
\end{equation*}
If $|z| \leqslant 2R$, then $|x+z| \leqslant 3R$. Hence, 
\begin{equation*}
    \left|\int_{|z| \leqslant2R} \nabla_{z}P_{t}(z)u(x+z)\,\ud z \right| \leqslant \int_{|z| \leqslant2R} \left| \nabla_{z}P_{t}(z)\right| u(x+z)\,\ud z\leqslant C_{\varepsilon} \|u\|_{L^{1}(B_{3R})}<\infty.
\end{equation*}
If $|z| >2R$, we obtain that
\begin{equation*}
    \left| \nabla_{z}P_{t}(z)\right| \leqslant\frac{C_{\varepsilon}}{\left(1+|z|\right)^{n+2s}} \leqslant\frac{C_{\varepsilon}}{1+|z|^{n+2s}} \leqslant\frac{2^{n+2s}C_{\varepsilon}}{1+|x+z|^{n+2s}}.
\end{equation*}
Consequently, 
\begin{equation*}
  \left|\int_{|z| >2R} \nabla_{z}P_{t}(z)u(x+z)\,\ud z\right|  \leqslant 2^{n+2s}C_{\varepsilon} \int_{|z|>2R}\frac{u(x+z)} {1+|x+z|^{n+2s}}\,\ud z <\infty.
\end{equation*}

In $\Omega_{2}$, we know $|t|\leqslant \frac{\varepsilon}{2}$. Since $|x|^{2}+t^{2} \geqslant \varepsilon^{2}$, we have $|x| \geqslant \delta:=\frac{\sqrt{3}}{2} \varepsilon$. Then $\frac{|x|}{2}\geqslant \frac{\delta}{2}$ for every $(x,t)\in B_R^+\setminus\overline{B_\varepsilon^+}$. If $|z|<\frac{|x|}{2}$, then $ \frac{\delta}{2}\leqslant|x+z|<\frac{3}{2}R$. Since $u \in \mathcal{C}^{1,1}_{\loc}(\mathbb{R}^{n} \setminus \{0\})$, there exists a constant $L>0$, depending only on $R$ and $\varepsilon$, such that $|u(x+z)-u(x)| \leqslant L|z|$. Hence, we obtain that
\begin{align*}
    \left|\int_{|z|<\frac{|x|}{2}}\nabla_{z}P_{t}(z)\left(u(x+z)-u(x)\right)\,\ud z \right|
    &\leqslant C \int_{|z|<\frac{|x|}{2}}\frac{t^{2s}|z|^{2}}{\left(|z|^{2}+t^{2}\right)^{\frac{n+2s+2}{2}}}\,\ud z \\
    &= C\int_{0}^{\frac{|x|}{2}} \frac{t^{2s}r^{n+1}}{\left(r^{2}+t^{2}\right)^{\frac{n+2s+2}{2}}}\,\ud r \\
     &\leqslant C \int_{0}^{\infty} \frac{\tau^{n+1}}{(1+\tau^{2})^{\frac{n+2s+2}{2}}}\,\ud \tau < \infty,
\end{align*}
where the change of variables $r=t\tau$ has been used in the last inequality. Furthermore,
\begin{equation*}
   \left| \int_{|z|<\frac{|x|}{2}}\nabla_{z}P_{t}(z)u(x+z)\,\ud z\right| \leqslant Ct^{-1} \int_{0}^{\infty}\frac{\tau^{n+1}}{(1+\tau^{2})^{\frac{n+2s+2}{2}}}\,\ud \tau .
\end{equation*}
For $\frac{|x|}{2} \leqslant|z| <2R$,
\begin{equation*}
    \left|\int_{\frac{|x|}{2} \leqslant|z|<2R}\nabla_{z}P_{t}(z)u(x+z)\,\ud z  \right|\leqslant C t^{2s}\|u\|_{L^{1} (B_{3R})}.
\end{equation*}
If $ |z| \geqslant 2R$, it follows that
\begin{equation*}
    \left|\int_{|z| \geqslant2R} \nabla_{z}P_{t}(z)u(x+z)\,\ud z\right|  \leqslant C t^{2s} \int_{|z|>3R}\frac{u(z)} {1+|z|^{n+2s}}\,\ud z.
\end{equation*}
Combining the above estimates and changing variables back to $y=x+z$, we conclude that for every fixed $(x,t)\in B_R^+\setminus\overline{B_\varepsilon^+}$, there exists a function $g\in L^1(\mathbb R^n)$ such that
\begin{equation*}
    \bigl|\nabla_xP_t(y-x)u(y)\bigr|
\leqslant g(y)
\quad
{\rm for\ a.e.}\ y\in\mathbb R^n.
\end{equation*}
Hence differentiation under the integral sign is justified. For $(x,t) \in \Omega$, we have
\begin{equation*}
    \nabla_xU(x,t)
=
\int_{\mathbb R^n}
\nabla_xP_t(y-x)u(y)\,\ud y=\int_{\mathbb R^n}
\nabla_xP_t(y-x)\left(u(y)-u(x)\right)\,\ud y,
\end{equation*}
where we have used \eqref{property of nablax P} in the last equality. Consequently, 
\begin{align*}
    \int_{\mathcal{B}_R^{+}\setminus \overline{\mathcal{B}_\varepsilon^{+}}}t^{1-2s}|\nabla_{x}U|^2\, \ud X =&\int_{\Omega_{1}}t^{1-2s}\left| \int_{\mathbb R^n}
\nabla_xP_t(y-x)u(y)\,\ud y\right|^{2}\,\ud X \\ &+ \int_{\Omega_{2}} t^{1-2s} \left| \int_{\mathbb R^n}
\nabla_xP_t(y-x)\left(u(y)-u(x)\right)\,\ud y\right|^{2}\,\ud X\\
    &\leqslant C\int_{\Omega_{1}}t^{1-2s}\,\ud X+C\int_{\Omega_{2}}t^{1-2s} \left(1+t^{4s} \right)\,\ud X <\infty.
\end{align*}
In the above estimate, we omit the case $(x,t)\in\Omega_2$ with
$|y-x|\geqslant |x|/2$ for brevity. A similar argument can be found in the proof of \eqref{sobolve energy of partialt U}.

Next, we prove that
\begin{equation}\label{sobolve energy of partialt U}
        \int_{\mathcal{B}_R^{+}\setminus \overline{\mathcal{B}_\varepsilon^{+}}}t^{1-2s}|\partial_{t}U|^2\,\ud X<\infty . 
        \end{equation}
 By \eqref{partialx P and partialt P}, we get
 \begin{equation*}
     |\partial_tP_t(z)|
\leqslant
C
\frac{t^{2s-1}}
{(|z|^2+t^2)^{\frac{n+2s}{2}}}, \quad {\rm and} \quad
\int_{\mathbb R^n}\partial_tP_t(z)\,\ud z=0.
 \end{equation*}
 Arguing as in the estimate of $\nabla_xU$, we obtain that for every $(x,t) \in \Omega$, 
 \begin{equation*}
     \partial_tU(x,t)
=\int_{\mathbb R^n}
\partial_tP_t(z)
u(x+z)\, \ud z=
\int_{\mathbb R^n}
\partial_tP_t(z)
\bigl(u(x+z)-u(x)\bigr)\,\ud z.
 \end{equation*}
Since the estimate in $\Omega_1$ for \eqref{sobolve energy of partialt U} can be obtained in exactly the same way as the proof of \eqref{sobolve energy of nablax U}, we only treat the case $\Omega_2$.

In $\Omega_{2}$, we know $|t|\leqslant \frac{\varepsilon}{2}$ and $|x| \geqslant \delta=\frac{\sqrt{3}}{2} \varepsilon$. Then $\frac{|x|}{2}\geqslant \frac{\delta}{2}$ for every $(x,t)\in B_R^+\setminus\overline{B_\varepsilon^+}$. Since $\delta \leqslant|x|\leqslant R$ and $u \in \mathcal{C}^{1,1}_{\loc}(\mathbb{R}^{n} \setminus \{0\})$, we have $|u(x)| \leqslant C$. If $|z|<\frac{|x|}{2}$, then $\frac{\delta}{2} \leqslant |x+z|<\frac{3}{2}R$. Since $u \in \mathcal{C}^{1,1}_{\loc}(\mathbb{R}^{n} \setminus \{0\})$, there exists a constant $L>0$, depending only on $R$ and $\varepsilon$, such that $|u(x+z)-u(x)| \leqslant L|z|$. Therefore,
\begin{align*}
    \int_{|z|<\frac{|x|}{2}}
\partial_tP_t(z)
\bigl(u(x+z)-u(x)\bigr)\,\ud z &\leqslant C\int_{|z|<\frac{|x|}{2}}\frac{t^{2s-1}|z|}
{(|z|^2+t^2)^{\frac{n+2s}{2}}}\,\ud z\\
&\leqslant \int_{0}^{\frac{|x|}{2t}} \frac{\tau^{n}}{\left(1+|\tau|^{2}\right)^{\frac{n+2s}{2}}}\,\ud \tau \leqslant Ct^{-(1-2s)}.
\end{align*}
If $\frac{|x|}{2}\leqslant|z|<2R$, then
\begin{align*}
    \int_{\frac{|x|}{2} \leqslant|z|<2R}\partial_{t}P_{t}(z)\bigl(u(x+z)-u(x)\bigr)\,\ud z &\leqslant C t^{2s-1}\left(\|u\|_{L^{1} (B_{3R})}+|u(x)|\int_{\frac{|x|}{2} \leqslant|z|<2R}\frac{1}{|z|^{n+2s}}\,\ud z\right) \\
    &\leqslant Ct^{2s-1}(\|u\|_{L^{1} (B_{3R})}+1).
\end{align*}
If $|z| \geqslant 2R$, we get
\begin{align*}
    \int_{|z|\geqslant2R}\partial_{t}P_{t}(z)\bigl(u(x+z)-u(x)\bigr)\,\ud z &\leqslant  C t^{2s-1} \left(\int_{|z|>3R}\frac{u(z)} {1+|z|^{n+2s}}\,\ud z+|u(x)|\int_{|z|>2R}\frac{1}{|z|^{n+2s}}\,\ud z\right)\\
    &\leqslant Ct^{2s-1}\left(\int_{|z|>3R}\frac{u(z)} {1+|z|^{n+2s}}\,\ud z+1\right).
\end{align*}
Combining the above estimates, we conclude that 
\begin{equation}
    \int_{\Omega_{2}}t^{1-2s}|\partial_{t}U|^2\,\ud X \leqslant C\int_{\Omega_{2}}t^{1-2s}\cdot t^{4s-2}\,\ud X<\infty,
\end{equation}
and Claim~3 is proved.

It remains to verify the Neumann condition in \eqref{eq:extension of u}. By the Caffarelli--Silvestre identity \cite{CaffarelliSilvestre2007}, the conormal derivative satisfies $\frac{\partial U}{\partial\nu^s}(x,0) = -\lim_{t\to 0^+} t^{1-2s}\partial_t U(x,t) = d_s(-\Delta)^s u(x)$ for every $x \neq 0$. Since $u$ solves \eqref{eq:frac Hartree}, this gives $\frac{\partial U}{\partial\nu^s}(x,0) = (\mathcal{R}_\alpha \ast u^{p_{\ast}}) u^{p_{\ast}-1}$, as claimed.
        \end{proof}

The following lemma shows that the Caffarelli--Silvestre extension is continuous in the upper half-space and up to the boundary away from the singular point, which will be needed in the subsequent application of the maximum principle. 

\begin{lemma}\label{U is continuous}
Let $n \geqslant 2$, $s \in (0,1)$, and $\alpha \in (0,n)$. If $u \in \mathcal{C}^{1,1}_{\loc}(\mathbb R^n\setminus\{0\})\cap L_s(\mathbb R^n) \cap L^{1}_{\loc}(\mathbb{R}^{n})$ is a positive solution to \eqref{eq:frac Hartree}, then its Caffarelli--Silvestre extension \(U\in \mathcal{C}^0(\mathbb R^{n+1}_+ \cup \mathbb{R}^{n}\setminus\{0\})\) defined by \eqref{U(x,t)} is continuous.
\end{lemma}

\begin{proof}
    Let $(x,t)\to (x_0,t_0)$ with $t_0>0$. Since $U$ is defined as \eqref{U(x,t)} and the integral is absolutely convergent by Lemma~\ref{Weak solution to the extension problem away from the origin}, the dominated convergence theorem yields $U(x,t)\to U(x_0,t_0)$. Hence $U\in \mathcal{C}(\mathbb R^{n+1}_+)$. 
    
    On the other hand, by Lemma~\ref{Weak solution to the extension problem away from the origin}, we know that $U(x,t)\to u(x)$ as $t\to0^{+}$ for every $x\neq0$. Since $u\in \mathcal{C}^{1,1}_{\loc}(\mathbb R^n\setminus\{0\})$ and $U(\cdot,t)=P_t\ast u$, where $P_t$ is the Poisson kernel for the extension operator, the convergence $P_t\ast u\to u$ as $t\to 0^+$ is locally uniform on compact subsets of $\mathbb R^n\setminus\{0\}$ by the approximate identity property of $P_t$. This yields the joint continuity: $U(x,t)\to u(x_0)$ whenever $(x,t)\to(x_0,0)$ and $x_0\neq0$. Therefore,
    \(U\in \mathcal{C}^0(\mathbb R^{n+1}_+\cup(\mathbb R^n\setminus\{0\}))\).
    \end{proof}
    Next, we recall a maximum principle for positive supersolutions with an isolated singularity.
    \begin{lemmaletter}[Proposition 3.1 in \cite{JinLiXiong2014}]\label{Proposition 3.1 in J}
    Let $n \geqslant 2$ and $s \in (0,1)$. Suppose that for all \(0<\varepsilon<R\) the extension $U\in W^{1,2}(t^{1-2s},\mathcal{B}_R^+\setminus \overline{\mathcal{B}_\varepsilon^+})$ is a solution to
    \begin{equation*}
        \begin{cases}
        -\operatorname{div}(t^{1-2s}\nabla U)\geqslant 0
        & {\rm in}\  \mathcal{B}_R^+,\\[2mm]
        \dfrac{\partial U}{\partial \nu^s}\geqslant 0
        & {\rm on}\  B_R\setminus \overline{B_\varepsilon},
        \end{cases}
    \end{equation*}
    If $U\in \mathcal{C}(\mathcal{B}_R^+\cup B_R\setminus\{0\})$ and $U>0$ in $\mathcal{B}_R^+\cup B_R\setminus\{0\}$, then
    \[
    \liminf_{X\to 0}U(X)>0.
    \]
    \end{lemmaletter}

   \subsection{Symmetry properties of the extension solution}
   In order to prove Theorem 1.1, it suffices to establish the corresponding symmetry property for the extension solution $U$. Indeed, since $u(x)=U(x,0)$, the desired symmetry of $u$ follows immediately from that of $U$.  

   \begin{proposition}
      Let $n \geqslant 2$, $s \in (0,1)$, and $\alpha \in (0,n)$. If $u \in \mathcal{C}^{1,1}_{\loc}(\mathbb R^n\setminus\{0\})\cap L_s(\mathbb R^n) \cap L^{1}_{\loc}(\mathbb{R}^{n})$ is a positive solution to \eqref{eq:frac Hartree} with a non-removable singularity at the origin, Then, its Caffarelli--Silvestre extension defined by \eqref{U(x,t)} satisfies $U(x,t)=U(|x|,t)$.
   \end{proposition}
   \begin{proof}
   The proof of this proposition is inspired by \cite{CabreSire2014,JinLiXiong2014,CaffarelliJinSireXiong2014}. It follows from Lemma~\ref{Proposition 3.1 in J} that
     \begin{equation*}
         \liminf_{|\xi| \to 0}U(\xi) >0.
     \end{equation*}
     Define 
     \begin{equation*}
         U_{X,\lambda}(Y):=\left(\frac{\lambda}{|Y-X|}\right)^{n-2s}U\left(Y^{X,\lambda}\right), \qquad Y^{X,\lambda}:=X+\frac{\lambda^{2}(Y-X)}{|Y-X|^{2}},     
    \end{equation*}
    and
    \begin{equation*}
        u_{x,\lambda}(y)=\left(\frac{\lambda}{|y-x|}\right)^{n-2s}u(y^{x,\lambda}),\qquad y^{x,\lambda} = x+\frac{\lambda^2(y-x)}{|y-x|^2},
    \end{equation*}
     where $X=(x,0)$. We first show that the moving sphere process can be initiated. More precisely, for all $x \in \mathbb{R}^n \setminus \{0\}$ there exists $\lambda_3(x)\in (0,|x|)$ such that for all $0<\lambda<\lambda_3(x)$ we have 
     \begin{equation*}
     U_{X,\lambda}(\xi)\leqslant\ U(\xi),
     \quad
     \forall\, |\xi-X|\geqslant\ \lambda,\ \xi\neq 0.
     \end{equation*}
     
     \noindent{\bf Claim~1:} {\it For sufficiently small $\lambda_1 \in (0,|X|)$, there exists $\lambda_2 = \lambda_2(\lambda_1) \in (0,\lambda_1)$ such that $U_{X,\lambda}(\xi) \leqslant U(\xi)$ for all $|\xi-X| \geqslant \lambda_1$ and $0 < \lambda < \lambda_2$.}

     \noindent{\it Proof.}
     Fix a sufficiently small $\lambda_1\in(0,|X|)$. We show that there exists $\lambda_2=\lambda_2(\lambda_1)\in(0,\lambda_1)$ such that
     \[
     U_{X,\lambda}(\xi)\leqslant U(\xi),
     \quad |\xi-X|\geqslant \lambda_1,
     \]
     for every $0<\lambda<\lambda_2$.

To this end, we define $\varphi(\xi)=(\lambda_{1}/|\xi-X|)^{n-2s} \inf_{\partial'' \mathcal{B}_{\lambda_{1}}(X)}U$, where 
\[
\partial'' \mathcal{B}_{\lambda_{1}}(X)=\{\xi \in \overline{\mathbb{R}^{n+1}_{+}} \,\big|\,|\xi-X|=\lambda_{1} \}.
\]
Since $|\xi-X|^{-(n-2s)}$ is the fundamental solution of $-\operatorname{div}(t^{1-2s}\nabla\,\cdot\,)$ in $\mathbb{R}^{n+1}$ (cf.\ \cite{FabesKenigSerapioni1982}), the function $\varphi$ satisfies
\begin{equation*}
\left\{
\begin{aligned}
&-\operatorname{div}\!\left(t^{1-2s}\nabla \varphi\right)=0
&&{\rm in}\  \mathbb{R}^{n+1}_{+}\setminus\mathcal{B}_{\lambda_{1}}^{+}(X) ,\\
&-\lim_{t\to0^{+}} t^{1-2s}\partial_t \varphi(x,t)=
0
&&{\rm on}\  \mathbb{R}^{n} \setminus \overline{B_{\lambda_{1}}(x)} ,
\end{aligned}
\right.
\end{equation*}
and $\varphi(\xi) \leqslant U(\xi)$ on $\partial'' \mathcal{B}_{\lambda_{1}}(X)$. Since $U - \varphi \geqslant 0$ on $\partial''\mathcal{B}_{\lambda_1}(X)$, the weak maximum principle (cf.\ \cite{CaffarelliJinSireXiong2014,JinLiXiong2014}) yields
\begin{equation*}
    U(\xi) \geqslant \left( \frac{\lambda_{1}}{|\xi-X|}\right)^{n-2s} \inf_{\partial'' \mathcal{B}_{\lambda_{1}}(X)} U, \qquad \forall |\xi-X| >\lambda_{1}, \,\xi \in \mathbb{R}^{n+1}_{+}.
\end{equation*}
Now set $\lambda_{2}=\lambda_{1}(\inf_{\partial'' \mathcal{B}_{\lambda_{1}}(X)} U / \sup_{\mathcal{B}_{\lambda_{1}}(X)}U) ^{1/(n-2s)}$. Then for any $0<\lambda<\lambda_2$ and $|\xi-X|\geqslant\lambda_1$,
\begin{equation*}
    U_{X,\lambda}(\xi)
=
\left(\frac{\lambda}{|\xi-X|}\right)^{n-2s}
U\!\left(
Y^{X,\lambda}
\right)
\leqslant
\left(\frac{\lambda_{2}}{|\xi-X|}\right)^{n-2s}
\sup_{\mathcal{B}_{\lambda_{1}}(X)}U
\leqslant
\left(\frac{\lambda_{1}}{|\xi-X|}\right)^{n-2s}
\inf_{\partial''\mathcal{B}_{\lambda_{1}}(X)}U
\leqslant U(\xi),
\end{equation*}
and Claim~1 is proved.

\medskip
\noindent{\bf Claim~2:} {\it There exists $0 < \lambda_3 < \lambda_1$ such that $U_{X,\lambda}(\xi) \leqslant U(\xi)$ for all $0 < \lambda < \lambda_3$ and $\lambda < |\xi-X| < \lambda_1$.}

\noindent{\it Proof.} We show that there exists $0<\lambda_{3}<\lambda_{1}$, such that
\begin{equation*}
    U_{X,\lambda}(\xi) \leqslant U(\xi), \qquad \forall 0<\lambda<\lambda_{3}, \, \lambda<|\xi-X| <\lambda_{1}.
\end{equation*}

For every $0<\lambda<\lambda_{3}<\lambda_{1}$ and $\xi \in \partial''\mathcal{B}_{\lambda_{1}}(X)$, we have $(X+\lambda^{2}(\xi-X)/|\xi-X|^{2}) \in \mathcal{B}_{\lambda_{1}}^{+}(X)$. Hence we can choose $\lambda_{3}=\lambda_{3}(\lambda_{1})<\lambda_{2}$ small such that 
\begin{equation*}
    U_{X,\lambda}(\xi)=
\left(\frac{\lambda}{|\xi-X|}\right)^{n-2s}
U\!\left(Y^{X,\lambda}\right) \leqslant\left(\frac{\lambda_{3}}{\lambda_{1}}\right)^{n-2s} \sup_{\overline{\mathcal{B}_{\lambda_{1}}^{+}(X)}}U \leqslant \inf_{\partial''\mathcal{B}_{\lambda_{1}}(X)} U \leqslant U(\xi).
\end{equation*}
Since $U_{X,\lambda}=U$ on $\partial''\mathcal{B}_{\lambda}(X)$, we obtain that $U_{X,\lambda} \leqslant U$ on $\partial'' (\mathcal{B}_{\lambda_{1}}(X) \setminus \mathcal{B}_{\lambda}(X))$. Note that the above choice of $\lambda_3$ still satisfies the requirement on $\lambda_2$ in Claim~1.

It remains to prove that $U_{X,\lambda}\leqslant U$ in $\mathcal{B}_{\lambda_1}^{+}(X)\setminus\mathcal{B}_{\lambda}^{+}(X)$ for sufficiently small $\lambda_1$ and all $0<\lambda<\lambda_3(\lambda_1)$. Since $u_{x,\lambda}$ also solves \eqref{eq:frac Hartree}, as can be verified by a change of variables in the Riesz integral using the criticality $p_{\ast} = (n+\alpha)/(n-2s)$; cf.\ \cite{CaffarelliJinSireXiong2014} for the Yamabe analog, $U_{X,\lambda}$ satisfies \eqref{eq:extension of u}. We therefore have
     \begin{equation*}
\left\{
\begin{aligned}
&-\operatorname{div}\!\left(t^{1-2s}\nabla \left(U-U_{X,\lambda} \right)\right)=0
&&{\rm in}\  \mathcal{B}_{\lambda_{1}}^{+}(X) \setminus \overline{\mathcal{B}_{\lambda}^{+}(X)} ,\\
&-\lim_{t\to0^{+}} t^{1-2s}\partial_t \left(U-U_{X,\lambda}\right)=
\left(\mathcal{R}_\alpha\ast u^{p_{*}}\right)u^{p_{*}-1}-\left(\mathcal{R}_\alpha\ast u_{x,\lambda}^{p_{*}}\right)u_{x,\lambda}^{p_{*}-1}
&&{\rm on}\  B_{\lambda_{1}}(x) \setminus \overline{B_{\lambda}(x)} ,
\end{aligned}
\right.
\end{equation*}
We shall make use of the idea of the narrow domain technique. Let $( U_{X,\lambda}-U)^{+}:=\max \{0, U_{X,\lambda}-U \}$. Taking $( U_{X,\lambda}-U)^{+}$ as a test function, we have
\begin{align}\label{I1+I2 in symmetry}\nonumber
    \int_{\mathcal{B}_{\lambda_{1}}^{+}(X) \setminus \mathcal{B}_{\lambda}^{+}(X)} & t^{1-2s} \Big|\nabla \left( U_{X,\lambda}-U\right)^{+} \Big|^{2} \\ \nonumber
    =&\int_{B_{\lambda_{1}}(x) \setminus B_{\lambda}(x)}\left(\left(\mathcal{R}_\alpha\ast u_{x,\lambda}^{p_{*}}\right)u_{x,\lambda}^{p_{*}-1}-\left(\mathcal{R}_\alpha\ast u^{p_{*}}\right)u^{p_{*}-1}\right) \left(u_{x,\lambda}-u\right)^{+}\\ \nonumber
    =&\int_{B_{\lambda_{1}}(x) \setminus B_{\lambda}(x)} \left(\mathcal{R}_\alpha\ast \left(u_{x,\lambda}^{p_{*}}-u^{p_{\ast}}\right)\right)u^{p_{\ast}-1}\left(u_{x,\lambda}-u\right)^{+}\\ \nonumber
    &+\int_{B_{\lambda_{1}}(x) \setminus B_{\lambda}(x)}\left(\mathcal{R}_\alpha\ast u_{x,\lambda}^{p_{*}}\right)\left(u_{x,\lambda}^{p_{\ast}-1}-u^{p_{\ast}-1}\right)\left(u_{x,\lambda}-u\right)^{+}\\
    :=&\mathcal{I}_{1}+\mathcal{I}_{2}.
\end{align}

For $\mathcal{I}_{1}$, by a standard Kelvin transform argument and the critical invariance and Claim~1, we have
\begin{align*}
    \mathcal{I}_{1}=&
\int_{B_{\lambda_{1}}(x)\setminus B_\lambda(x)}
u^{p_{\ast}-1}(y)(u_{x,\lambda}-u)^+(y)
\int_{\{|z-x|\geqslant\lambda\}}
K_\lambda(y,z)
\bigl(u_{x,\lambda}^{p_{\ast}}(z)-u^{p_{\ast}}(z)\bigr)
\,\ud z\,\ud y\\
&\leqslant\int_{B_{\lambda_{1}}(x)\setminus B_\lambda(x)}\int_{B_{\lambda_{1}}(x)\setminus B_\lambda(x)}\frac{u_{x,\lambda}^{p_{\ast}-1}(y)(u_{x,\lambda}-u)^+(y) \left(u_{x,\lambda}^{p_{\ast}}(z)-u^{p_{\ast}}(z) \right)}{|y-z|^{n-\alpha}}\,\ud z\,\ud y,
\end{align*}
where
\begin{equation*}
    K_\lambda(y,z):=\frac1{|y-z|^{n-\alpha}}
-
\left(\frac{\lambda}{|y-x|}\right)^{n-\alpha}
\frac1{|y^{x,\lambda}-z|^{n-\alpha}}
\end{equation*}
and $K_\lambda(y,z)>0$ for all $|y-x| >\lambda$ and $|z-x|> \lambda$. In the preceding inequality, the integral over $\{|z-x|>\lambda_1\}$ has been dropped since $u_{x,\lambda}^{p_{\ast}} \leqslant u^{p_{\ast}}$ there by Claim~1, making the integrand nonpositive. Now, on the support of $(u_{x,\lambda}-u)^+$ one has $u^{p_{\ast}-1} \leqslant u_{x,\lambda}^{p_{\ast}-1}$, and the mean value theorem gives $u_{x,\lambda}^{p_{\ast}}(z) - u^{p_{\ast}}(z) \leqslant p_{\ast} u_{x,\lambda}^{p_{\ast}-1}(z)(u_{x,\lambda}(z) - u(z))^+$. Therefore,
\begin{align*}
    \mathcal{I}_{1} &\leqslant C\int_{B_{\lambda_{1}}(x)\setminus B_\lambda(x)}\int_{B_{\lambda_{1}}(x)\setminus B_\lambda(x)}\frac{u_{x,\lambda}^{p_{\ast}-1}(y)(u_{x,\lambda}-u)^+(y) u_{x,\lambda}^{p_{\ast}-1}(z)(u_{x,\lambda}-u)^+(z)}{|y-z|^{n-\alpha}}\,\ud z\,\ud y \\
    & \leqslant C \left(\int_{B_{\lambda_{1}}(x)\setminus B_\lambda(x)}\left(\left(u_{x,\lambda}-u\right)^{+}(y)\right)^{\frac{2n}{n+\alpha}}u_{x,\lambda}^{\frac{2n(\alpha+2s)}{(n-2s)(n+\alpha)}}(y) \,\ud y \right)^{\frac{n+\alpha}{n}}\\
    & \leqslant C \left(\int_{B_{\lambda_{1}}(x)\setminus B_\lambda(x)} \left(\left( u_{x,\lambda}-u\right)^{+}(y) \right)^{\frac{2n}{n-2s}}\,\ud y\right)^{\frac{n-2s}{n}} \left( \int_{B_{\lambda_{1}}(x)\setminus B_\lambda(x)} u_{x,\lambda}^{\frac{2n}{n-2s}}(y) \,\ud y\right)^{\frac{\alpha+2s}{n}} \\
    & \leqslant C \left(\int_{B_{\lambda_{1}}(x)\setminus B_\lambda(x)} \left(\left( u_{x,\lambda}-u\right)^{+}(y) \right)^{\frac{2n}{n-2s}}\,\ud y\right)^{\frac{n-2s}{n}} \left( \int_{B_{\lambda_{1}(x)}} u^{\frac{2n}{n-2s}}(y) \,\ud y\right)^{\frac{\alpha+2s}{n}},
\end{align*}
where we have used the Hardy--Littlewood--Sobolev inequality and H\"older's inequality in the second and the last steps, respectively. By the weighted Sobolev trace inequality (cf.\ \cite{CaffarelliSilvestre2007, CabreSire2014}), we obtain that
\begin{equation}\label{I1 in symmetry}
    \mathcal{I}_{1} \leqslant C \left( \int_{\mathcal{B}_{\lambda_{1}}^{+}(X) \setminus \mathcal{B}_{\lambda}^{+}(X)} t^{1-2s} \Big| \nabla \left(U_{X,\lambda}-U\right)^{+} \Big|^{2}  \right) \left( \int_{B_{\lambda_{1}(x)}} u^{\frac{2n}{n-2s}}(y) \,\ud y\right)^{\frac{\alpha+2s}{n}}.
\end{equation}

For $\mathcal{I}_{2}$, using the Kelvin transform and the inversion distance identity
\begin{equation*}
    |y-z|=\frac{|y-x|\,|z-x|}{\lambda^{2}}\,|y^{x,\lambda}-z^{x,\lambda}|,
\end{equation*}
we have
\begin{equation*}
    (\mathcal{R}_\alpha*u_{x,\lambda}^{p_{\ast}})(y)
=
\left(\frac{\lambda}{|y-x|}\right)^{n-\alpha}
(\mathcal{R}_\alpha*u^{p_{\ast}})(y^{x,\lambda}) \leqslant \sup_{  B_{\lambda_{1}}(x)} (\mathcal{R}_\alpha*u^{p_{\ast}}) .
\end{equation*}
This, together with the mean value theorem, yields
\begin{align*}
    \mathcal{I}_{2} &\leqslant C\left(\sup_{B_{\lambda_{1}}(x)} (\mathcal{R}_\alpha*u^{p_{\ast}}) \right)\int_{B_{\lambda_{1}}(x)\setminus B_\lambda(x)} \max \{u^{p_{\ast}-2}, u_{x,\lambda}^{p_{\ast}-2} \}(y)\left( \left(u_{x,\lambda}-u\right)^{+}  (y)\right)^{2}\,\ud y\\
    &\leqslant C \left(\sup_{B_{\lambda_{1}}(x)} (\mathcal{R}_\alpha*u^{p_{\ast}}) \right) \left(\sup_{B_{\lambda_{1}}(x)}u^{p_{\ast}-2}\right) \int_{B_{\lambda_{1}}(x)\setminus B_\lambda(x)}\left( \left(u_{x,\lambda}-u\right)^{+}  (y)\right)^{2}\,\ud y,
\end{align*}
where we used the definition of $u_{x,\lambda}$ and the fact that
\begin{equation*}
    \max\{u^{p_{\ast}-2},u_{x,\lambda}^{p_{\ast}-2}\}
=
\begin{cases}
u_{x,\lambda}^{p_{\ast}-2}, & p_{\ast}\geqslant 2,\\
u^{p_{\ast}-2}, & 1<p_{\ast}<2,
\end{cases}
\end{equation*}
 on the support of $(u_{x,\lambda}-u)^+$. Then, by H\"older's inequality and the weighted Sobolev trace inequality (cf.\ \cite{CaffarelliSilvestre2007, CabreSire2014}), we obtain that
\begin{align}\label{I2 in symmetry}
\mathcal{I}_{2}
&\leqslant
C
\Bigl(\sup_{B_{\lambda_1}(x)}
(\mathcal{R}_\alpha*u^{p_{\ast}})\Bigr)
\Bigl(\sup_{B_{\lambda_1}(x)}
u^{p_{\ast}-2}\Bigr)
|B_{\lambda_1}(x)\setminus B_\lambda(x)|^{\frac{2s}{n}}
\int_{\mathcal{B}_{\lambda_1}^+(X)\setminus \mathcal{B}_\lambda^+(X)}
t^{1-2s}
\left|\nabla (U_{x,\lambda}-U)^+\right|^2.
\end{align}
By \eqref{boundedness of the convolution term}, the positivity of \(u\), and the fact that
\(u\in \mathcal{C}^{1,1}_{\loc}(\mathbb R^n\setminus\{0\})\), for \(\lambda_1>0\) sufficiently small, we have
\begin{equation}\label{Cx}
    C\left(\sup_{B_{\lambda_{1}}(x)} (\mathcal{R}_\alpha*u^{p_{\ast}}) \right) \left(\sup_{B_{\lambda_{1}}(x)}u^{p_{\ast}-2}\right) \leqslant C\sup_{B_{|x|/2}(x)}\left(\left( \mathcal{R}_\alpha*u^{p_{\ast}}\right)u^{p_{\ast}-2} \right)=:C_x<\infty,
\end{equation}
where $C_x$ is independent of $\lambda_1$.

It follows from \eqref{I1+I2 in symmetry}, \eqref{I1 in symmetry}, \eqref{I2 in symmetry}, and \eqref{Cx} that
\begin{equation*}
    \int_{\mathcal{B}_{\lambda_{1}}^{+}(X) \setminus \mathcal{B}_{\lambda}^{+}(X)} t^{1-2s} \Big|\nabla \left( U_{X,\lambda}-U\right)^{+} \Big|^{2} \leqslant C_{\lambda_{1}} \int_{\mathcal{B}_{\lambda_{1}}^{+}(X) \setminus \mathcal{B}_{\lambda}^{+}(X)} t^{1-2s} \Big|\nabla \left( U_{X,\lambda}-U\right)^{+} \Big|^{2},
\end{equation*}
where
\begin{equation*}
    C_{\lambda_{1}}=C\left( \int_{B_{\lambda_{1}(x)}} u^{\frac{2n}{n-2s}}(y) \,\ud y\right)^{\frac{\alpha+2s}{n}}+ C_{x}
|B_{\lambda_1}(x)\setminus B_\lambda(x)|^{\frac{2s}{n}}.
\end{equation*}
Here, $C$ is a positive constant depending only on $n, \alpha$ and $s$. Since $u \in \mathcal{C}^{1,1}_{\loc}(\mathbb{R}^{n} \setminus \{0\})$, we can choose $\lambda_{1}>0$ sufficiently small such that $C_{\lambda_{1}}<1$. Then,
\begin{equation*}
    \nabla \left( U_{X,\lambda}-U\right)^{+}=0 \quad {\rm in}\ \mathcal{B}_{\lambda_1}^+(X)\setminus \mathcal{B}_\lambda^+(X).
\end{equation*}
Since
\begin{equation*}
    \left( U_{X,\lambda}-U\right)^{+}=0 \quad {\rm on}\  \partial''\left(\mathcal{B}_{\lambda_{1}}(X) \setminus \mathcal{B}_{\lambda}(X)\right),
\end{equation*}
we have
\begin{equation*}
    \left( U_{X,\lambda}-U\right)^{+}=0 \quad {\rm in}\ \mathcal{B}_{\lambda_1}^+(X)\setminus \mathcal{B}_\lambda^+(X).
\end{equation*}
Hence Claim~2 is proved.

     Now we define 
     \begin{equation*}
        \bar{\lambda}(x):=\sup\Bigl\{0<\mu\leqslant |x|\;\Big|\;U_{X,\lambda}(\xi)\leqslant U(\xi),\quad\forall\, |\xi-X|\geqslant\lambda,\ \xi\neq 0,\quad\forall\, 0<\lambda<\mu\Bigr\}.
     \end{equation*}
     
     \medskip
     \noindent{\bf Claim~3:} {\it $\bar{\lambda}(x) = |x|$ for every $x \neq 0$.}

     \noindent{\it Proof.}
     Suppose that $\bar{\lambda}(x)<|x|$ for some $x \neq 0$. Since $0$ is a non-removable singularity, the case $U\equiv U_{X,\bar{\lambda}}$ is impossible. Hence, by the strong maximum principle (cf.~\cite{CabreSire2014}), we have
     \begin{equation*}
         U(\xi)>U_{X,\bar{\lambda}}(\xi),
\qquad
\xi\in\overline{\mathbb{R}^{n+1}_+}\setminus\{0\},
\ |\xi-X|>\bar{\lambda}.
     \end{equation*}
Then by Lemma~\ref{U is continuous} and Lemma~\ref{Proposition 3.1 in J}, we have 
     \begin{equation*}
         \liminf_{\xi \to 0} \left( U(\xi)-U_{X,\bar{\lambda}}(\xi)\right)>0.
     \end{equation*}
     For $\delta>0$ small, which will be fixed later, we denote 
     \[
     K_{\delta}= \left\{ \xi \in \overline{\mathbb{R}^{n+1}_{+}} : |\xi-X| \geqslant \bar{\lambda}+\delta \quad {\rm and} \quad \xi\neq 0  \right\}.
     \]
     By the uniform continuity of $U$ on compact sets, there exists $\varepsilon$ small such that for $\bar{\lambda}<\lambda<\bar{\lambda}+\varepsilon$, we have 
     \begin{equation*}
         U-U_{X,\lambda}>0 \quad {\rm in} \ K_{\delta}.
     \end{equation*}
     Now, let us focus on the region
     \[
     \widehat{K}_{\delta}= \left\{ \xi \in \overline{\mathbb{R}^{n+1}_{+}} : \lambda\leqslant|\xi-X| < \bar{\lambda}+\delta  \right\}.
     \]
     
     Using the narrow domain technique and arguing as in the proof of Claim~2, we can choose $\varepsilon$ and $\delta$ sufficiently small such that
     \begin{equation*}
          U_{X,\lambda}(\xi) \leqslant U(\xi) \quad {\rm in} \ \widehat{K}_{\delta}.
     \end{equation*}
     In conclusion, there exists $\varepsilon_{1}>0$ such that for all $\bar{\lambda}<\lambda<\bar{\lambda}+\varepsilon_{1}$,
     \begin{equation*}
         U_{X,\lambda}(\xi) \leqslant U(\xi) \quad {\rm for\ all} \quad |\xi-X| \geqslant\lambda, \quad \xi \neq0, \quad {\rm and} \quad \xi \in \overline{\mathbb{R}^{n+1}_{+}},
     \end{equation*}
which contradicts the definition of $\bar{\lambda}$, and Claim~3 is proved.

Hence, we have 
\begin{equation}\label{moving sphere inequality}
U_{X,\lambda}(\xi)\leqslant U(\xi),
\quad {\rm for\ all} \quad |\xi-X| \geqslant\lambda, \quad \xi \neq0, \quad {\rm and} \quad 0<\lambda<|x|.
\end{equation}
For any unit vector $e\in\mathbb R^n$, any $a>0$, and any
$\xi=(y,t)\in\mathbb R^{n+1}_+$ satisfying
$(\xi-ae)\cdot e<0$, we choose $x=Re$ and $\lambda=R-a$ in \eqref{moving sphere inequality}. By a standard computation, the Kelvin inversion centered at $X=Re$ with radius $\lambda=R-a$ converges to the reflection across the hyperplane $\{y\cdot e=a\}$ as $R\to\infty$. Letting $R\to\infty$, we obtain
\[
U(y,t) \geqslant U(y-2(y \cdot e-a)e,t).
\]
Since the unit vector $e$ and $a>0$ are arbitrary, $U(\cdot,t)$ is symmetric across every hyperplane through the origin in the $y$-variable, and hence $U$ is radially symmetric with respect to the $y$-variable and nonincreasing in $|y|$. Consequently, $u$ is radially symmetric about the origin and nonincreasing with respect to $r=|x|$.
 \end{proof}

    \section{Reformulation for Delaunay-type solutions}\label{sec:prelim}

    In this section, we begin the study of the existence of Delaunay-type solutions to equation \eqref{eq:frac Hartree}. We first reformulate the problem through the Emden--Fowler transformation and set up the variational framework. Throughout this section, we restrict our attention to radially symmetric solutions.

	\subsection{Periodic formulation of the problem}
	We consider the singular fractional Hartree equation \eqref{eq:frac Hartree} with $n>2s$, $s \in (0,1)$ and $\alpha \in (0,n)$. We consider radially symmetric solutions of the form
	\begin{equation}\label{form of u}
		u(x)=|x|^{-\frac{n-2s}{2}}v(|x|).
	\end{equation}
    Using polar coordinates $x = r\theta$ and $y = \rho\sigma$ with $r=|x|$, $\rho=|y|$ and $\theta,\sigma \in \mathbb{S}^{n-1}$, the fractional Laplacian in \eqref{fraction laplace} can be rewritten as
	\begin{equation*}
		\left(-\Delta\right)^{s}u(x)=\kappa_{n,s}\, \mathrm{P.V.}\int_{0}^{\infty}\int_{\mathbb{S}^{n-1}}\frac{r^{-\frac{n-2s}{2}}v(r)-\rho^{-\frac{n-2s}{2}}v(\rho)}{\left|r^{2}+\rho^{2}-2r\rho\langle \theta, \sigma \rangle\right|^{\frac{n+2s}{2}}}\rho^{n-1}\,\ud \sigma\,\ud \rho.
	\end{equation*}
	Let $\rho = r\bar{\rho}$ and apply the Emden--Fowler transformation
	$r = e^{t}$ and $\rho = e^{\tau}$. Then the equation \eqref{eq:frac Hartree} for $v$ becomes
	\begin{equation}\label{eq:Lgamma=hartree}
		\mathcal{L}_{s}v(t)=\kappa_{n,s}\,\mathrm{P.V.}
		\int_{-\infty}^{+\infty}
		\bigl(v(t)-v(\tau)\bigr)
		\mathcal{K}_{s}(t-\tau) \ud\tau+Av(t)=\left(\int_{-\infty}^{+\infty}\mathcal{K}_{-\frac{\alpha}{2}}(t-\tau)v(\tau)^{p_{\ast}} \ud\tau\right)v(t)^{p_{\ast}-1}.
	\end{equation}
	The left-hand side is derived as in \cite{MR3694655}. For the right-hand side, we write $(\mathcal{R}_\alpha * u^{p_{\ast}})(e^t\theta)$ in polar coordinates $y = e^\tau\sigma$ with $\sigma \in \mathbb{S}^{n-1}$ and substitute \eqref{form of u}. Factoring $e^{t+\tau}$ from the Riesz kernel $|e^t\theta - e^\tau\sigma|^{-(n-\alpha)}$, the exponential prefactors in $\tau$ combine to give a total exponent
	\[
	n - \tfrac{(n-2s)p_{\ast}}{2} - \tfrac{n-\alpha}{2} = 0,
	\]
	since $p_{\ast} = (n+\alpha)/(n-2s)$, so only the kernel $\mathcal{K}_{-\frac{\alpha}{2}}(t-\tau)$ and a factor $e^{-t(n-\alpha)/2}$ remain. Multiplying by $u(e^t\theta)^{p_{\ast}-1} = e^{-t\frac{(n-2s)(p_{\ast}-1)}{2}} v(t)^{p_{\ast}-1}$ and observing that $\tfrac{n-\alpha}{2} + \tfrac{(n-2s)(p_{\ast}-1)}{2} = \tfrac{n+2s}{2}$ matches the exponent from the left-hand side, the equation \eqref{eq:Lgamma=hartree} follows.

	In \eqref{eq:Lgamma=hartree},
	\begin{equation*}
		A = \kappa_{n,s}\,\mathrm{P.V.}\!
		\int_{0}^{\infty}\!\int_{\mathbb{S}^{n-1}}
		\frac{\left(1-\bar{\rho}^{-\frac{n-2s}{2}}\right)\bar{\rho}^{\,n-1}}
		{\left|1+\bar{\rho}^{2}-2\bar{\rho}\langle \theta,\sigma\rangle\right|^{\frac{n+2s}{2}}}
		\, \ud\sigma\, \ud\bar{\rho}=c_{n,
			s}>0
		\quad \text{(cf.\ \cite[Lemma~3.2]{MR3694655})},
	\end{equation*}
	and the kernel $\mathcal{K}_{m}$ for $-\frac{n}{2}<m<1$ is defined by
	\begin{equation}\label{Km}
		\mathcal{K}_{m}(t)=2^{-\frac{n+2m}{2}}
		\int_{\mathbb{S}^{n-1}}
		\frac{1}{\left|\cosh(t)-\langle\theta,\sigma\rangle\right|^{\frac{n+2m}{2}}}
		\, \ud\sigma
		=
		\int_{\mathbb{S}^{n-1}}
		\frac{1}
		{\left(e^{t}+e^{-t}-2\langle\theta,\sigma\rangle\right)^{\frac{n+2m}{2}}}
		\, \ud\sigma .
	\end{equation}
	
	The following lemma collects the key properties of the kernel $\mathcal{K}_m$ that will be used throughout the paper. Since the equation involves only the values $m=-\alpha/2$ and $m=s$, we state the result for $\mathcal{K}_{-\frac{\alpha}{2}}$ with $\alpha\in(0,n)$ and $\mathcal{K}_{s}$ with $s\in(0,1)$.
	\begin{lemma}[Properties of the kernels]\label{Pro of ker}
		Let $n \geqslant 2$, $s \in (0,1)$, and $\alpha \in (0,n)$, and let $\mathcal{K}_m$ be defined by \eqref{Km}. The following properties hold.
		\begin{enumerate}
			\item[{\rm (i)}] The kernel $\mathcal{K}_{-\frac{\alpha}{2}}$, with $\alpha \in (0,n)$, satisfies the following.
			\begin{enumerate}
				\item[{\rm (a)}] $\mathcal{K}_{-\frac{\alpha}{2}}$ is strictly positive, even, and belongs to $L^{1}(\mathbb{R})$.
				\item[{\rm (b)}] As $|t|\to+\infty$, one has
				\[
				\mathcal{K}_{-\frac{\alpha}{2}}(t)=O\!\left(e^{-\frac{n-\alpha}{2}|t|}\right).
				\]
				\item[{\rm (c)}] As $|t|\to0$, the singularity depends on $\alpha$. If $\alpha\in(0,1)$, then
				\[
				\mathcal{K}_{-\frac{\alpha}{2}}(t)\sim |t|^{\alpha-1}.
				\]
				If $\alpha=1$, then $\mathcal{K}_{-\frac{1}{2}}(t)\sim |\ln|t||$.
				\item[{\rm (d)}] If $\alpha>1$, then $\mathcal{K}_{-\frac{\alpha}{2}}$ is bounded and H\"older continuous on $\mathbb{R}$.
			\end{enumerate}

			\item[{\rm (ii)}] The kernel $\mathcal{K}_{s}$, with $s \in (0,1)$, satisfies the following.
			\begin{enumerate}
				\item[{\rm (a)}] $\mathcal{K}_{s}$ is strictly positive and even.
				\item[{\rm (b)}] As $|t|\to+\infty$, one has
				\[
				\mathcal{K}_{s}(t)
				= O\!\left(e^{-\frac{n+2s}{2}|t|}\right).
				\]
				\item[{\rm (c)}] As $|t|\to0$, one has
				\[
				\mathcal{K}_{s}(t)\sim |t|^{-1-2s}.
				\]
			\end{enumerate}
		\end{enumerate}
	\end{lemma}
	\begin{proof}
		The qualitative properties {\rm (i)(b)--(d)} and {\rm (ii)(a)--(c)} are established in \cite[Lemma~2.1]{JinXiong2020} and \cite[Lemma~3.3]{MR3694655}. The positivity and evenness of $\mathcal{K}_{-\frac{\alpha}{2}}$ follow from the same references. It remains to verify the global integrability $\mathcal{K}_{-\frac{\alpha}{2}}\in L^1(\mathbb{R})$, which we now prove.

		Since $\mathcal{K}_{-\frac{\alpha}{2}}$ is even, it suffices to show that $\mathcal{K}_{-\frac{\alpha}{2}}\in L^1(0,+\infty)$, and we split the integral into the contributions near the origin and at infinity.

		For the tail, fix $\alpha\in(0,n)$. It follows from {\rm (i)(b)} that there exist constants $C,M>0$ such that $\mathcal{K}_{-\frac{\alpha}{2}}(t)\leqslant C e^{-\frac{n-\alpha}{2}t}$ for all $t\geqslant M$. Hence
		\[
			\int_{M}^{+\infty} \mathcal{K}_{-\frac{\alpha}{2}}(t)\,\ud t
			\leqslant C \int_{M}^{+\infty} e^{-\frac{n-\alpha}{2}t}\,\ud t
			= \frac{2C}{n-\alpha}\,e^{-\frac{n-\alpha}{2}M}
			< +\infty.
		\]

		For the behavior near the origin, it follows from {\rm (i)(c)--(d)} that, as $t\to0^+$,
		\[
			\mathcal{K}_{-\frac{\alpha}{2}}(t)\sim
			\begin{cases}
				t^{\alpha-1}, & 0<\alpha<1,\\
				|\ln t|, & \alpha=1,\\
				1, & 1<\alpha<n.
			\end{cases}
		\]
		In each case, the right-hand side is integrable on $(0,1)$, so that $\mathcal{K}_{-\frac{\alpha}{2}}\in L^1(0,1)$.

		Combining the two estimates, we conclude that $\mathcal{K}_{-\frac{\alpha}{2}}\in L^1(0,+\infty)$, and therefore $\mathcal{K}_{-\frac{\alpha}{2}}\in L^1(\mathbb{R})$.
	\end{proof}
	We are looking for periodic solutions of \eqref{eq:Lgamma=hartree}. Assume that $v(t+L)=v(t)$. In this case, equation \eqref{eq:Lgamma=hartree} becomes
	\begin{equation}\label{eq:periodic hartree 2}
		\mathcal{L}_{s}^{L} v(t)=\kappa_{n,s}\,\mathrm{P.V.}
		\int_{0}^{L}\bigl(v(t)-v(\tau)\bigr)\mathcal{K}_{s}^{L}(t-\tau)\,\ud\tau+c_{n,s} v(t)=\left(\int_{0}^{L}\mathcal{K}_{-\frac{\alpha}{2}}^{L}(t-\tau)\, v(\tau)^{p_{\ast}}\, \ud\tau\right)v(t)^{p_{\ast}-1},
	\end{equation}
	and
	\begin{equation}\label{Kml}
		\mathcal{K}_{m}^{L}(t)=\sum_{j\in\mathbb{Z}}\mathcal{K}_{m}\left(t+jL\right)
	\end{equation}
	for the kernel given in \eqref{Km}. By Lemma~\ref{Pro of ker}, we know that $\mathcal{K}_{-\frac{\alpha}{2}}^{L} \in L^{1}(0,L)$ for $\alpha \in (0,n)$.
	
	\subsection{Extension problem}
	We restrict attention to radially symmetric
	solutions of the fractional Hartree equation
	\eqref{eq:frac Hartree}. As explained in the
	previous subsection, after applying the scaling
	transformation \eqref{scaling} together with the
	Emden--Fowler transformation $r=e^{t}$, the problem \eqref{eq:frac Hartree} can be
	reformulated as \eqref{eq:Lgamma=hartree} on the
	cylindrical manifold
	\[
	M=\mathbb{R}\times \mathbb{S}^{n-1},
	\]
	where the original scaling invariance of
	\eqref{eq:frac Hartree} is converted into
	translation invariance in the variable $t$.

	To obtain a local realization of the nonlocal operator, we employ the Caffarelli--Silvestre extension, which was originally introduced by Caffarelli and Silvestre \cite{CaffarelliSilvestre2007} and later developed in the geometric setting by
	Chang and Gonz\'alez \cite{CG11} and by Case and
	Chang \cite{CaseChang16}. This extension represents the
	fractional operator as the Dirichlet-to-Neumann
	map of a local degenerate elliptic equation in
	one higher dimension.
	
	Inspired by DelaTorre and González~\cite{DelaTorreGonzalez18}, let
	\(X^{n+1}=M\times(0,2)\) 	be endowed with a metric $\bar g$ associated with a defining function $\rho \in (0,2)$. For a boundary function
	$v$ on $M$, let $V=V(t,\rho)$ denote its extension
	to $X^{n+1}$. In general, for radial functions, the dependence
	on the angular variable disappears, and equation
	\eqref{eq:Lgamma=hartree} is equivalent to the
	corresponding extension equation
	\begin{equation}\label{eq:extension-general}
		\begin{cases}
			-\operatorname{div}_{\bar g}\!\big(\rho^{1-2s}
			\nabla_{\bar g} V\big)
			+ E(\rho)\,V = 0
			& {\rm in}\ (X^{n+1},\bar g),\\[6pt]
			
			V = v
			& {\rm on}\ \{\rho = 0\},\\[6pt]
			
			- d_s \displaystyle\lim_{\rho\to0}
			\rho^{1-2s}\partial_\rho V
			=
			\left(\int_{-\infty}^{+\infty}\mathcal{K}_{-\frac{\alpha}{2}}(t-\tau) v(\tau)^{p_{\ast}}\,\ud \tau\right)
			v(t)^{p_{\ast}-1}
			& {\rm on}\ \{\rho = 0\},
		\end{cases}
	\end{equation}
	where the lower-order term $E(\rho)V$ arises from
	the curvature of the ambient metric and $d_{s} = 2^{1-2s}\Gamma(1-s)/\Gamma(s)$ is a constant depending on $s$.
	
	\medskip
	
	Following \cite{CG11}, one can choose a
	suitable geodesic defining function $\rho^*=\rho^*(\rho)$ so that problem~\eqref{eq:extension-general} can be rewritten on the extension $X^{\ast}=M \times (0,\rho_{0}^{\ast})$, where $\rho_0^{\ast}$ depends on the geometry of $M$, and the lower-order term is eliminated. In terms
	of this normalized defining function, the extension
	problem becomes
	\begin{equation}\label{eq:extension-normalized}
		\begin{cases}
			-\operatorname{div}_{g^*}\!\big((\rho^*)^{1-2s}
			\nabla_{g^*} V\big)=0
			& {\rm in}\ (X^*,g^*),\\[6pt]
			
			V = v
			& {\rm on}\ \{\rho^* = 0\},\\[6pt]
			
			- d_s \displaystyle\lim_{\rho^*\to0}
			(\rho^*)^{1-2s}\partial_{\rho^*}V
			+ c_{n,s}\,v
			=
			\left(\int_{-\infty}^{+\infty}\mathcal{K}_{-\frac{\alpha}{2}}(t-\tau) v(\tau)^{p_{\ast}}\,\ud \tau\right)
			v(t)^{p_{\ast}-1}
			& {\rm on}\ \{\rho^* = 0\},
		\end{cases}
	\end{equation}
	where $g^{\ast}=(\rho^{\ast})^{2}\rho^{-2}\bar{g}$.
	The operator
	\[
	\mathcal L_s v
	=
	- d_s \lim_{\rho^*\to0}
	(\rho^*)^{1-2s}\partial_{\rho^*}V
	+ c_{n,s}\,v
	\]
	defines the Dirichlet-to-Neumann map associated
	with the normalized extension problem.
	
	Under the radial assumption, if we further
	consider periodic solutions satisfying
	$v(t+L)=v(t)$, then the equation \eqref{eq:extension-normalized} reduces to
	\begin{equation}\label{eq:periodic extension-normalized}
		\begin{cases}
			-\operatorname{div}_{g^*}\!\big((\rho^*)^{1-2s}
			\nabla_{g^*} V\big)=0
			& {\rm in}\ (X^*,g^*),\\[6pt]
			
			V = v
			& {\rm on}\ \{\rho^* = 0\},\\[6pt]
			
			- d_s \displaystyle\lim_{\rho^*\to0}
			(\rho^*)^{1-2s}\partial_{\rho^*}V
			+ c_{n,s}\,v
			=
			\left(\int_{0}^{L}\mathcal{K}_{-\frac{\alpha}{2}}^{L}(t-\tau)\, v(\tau)^{p_{\ast}}\, \ud\tau\right)v(t)^{p_{\ast}-1}
			& {\rm on}\ \{\rho^* = 0\}.
		\end{cases}
	\end{equation}
	
	In the subsequent sections, we exploit the
	equivalent extension formulation
	\eqref{eq:periodic extension-normalized} to
	investigate the local regularity of solutions of \eqref{eq:cylinder-Hartree-periodic}.
	\section{Regularity}\label{sec:regularity}
	In this section, we develop the analytical framework required for the variational approach. We first introduce the function spaces associated with the periodic problem and prove the corresponding compact embedding. We then investigate the regularity of weak solutions, which will be used in the proof of the existence of minimizers in the next section.

	\subsection{Function spaces}
	We begin by introducing the weighted Sobolev space associated with the periodic extension problem and establishing the compact trace embedding that will be used in the variational argument.

	\begin{definition}
		We consider the following function space
		\begin{multline*}
			H_L^{s}
			=
			\biggl\{
			v:\mathbb{R}\to\mathbb{R}\; ;\;
			v(t+L)=v(t)
			\ {\rm and} \ \\
			\int_{0}^{L}\!\!\int_{0}^{L}
			\bigl(v(t)-v(\tau)\bigr)^2
			\mathcal{K}_{s}^{L}(t-\tau)\,\ud t\,\ud\tau
			+
			\int_{0}^{L} v(t)^2\,\ud t
			<+\infty
			\biggr\}.
		\end{multline*}
		This space is endowed with the norm
		\begin{equation*}
			\|v\|_{H_L^{s}}
			=
			\left(
			\int_{0}^{L} v(t)^2\,\ud t
			+
			\int_{0}^{L}\!\!\int_{0}^{L}
			|v(t)-v(\tau)|^2
			\mathcal{K}_{s}^{L}(t-\tau)\,\ud t\,\ud\tau
			\right)^{1/2}.
		\end{equation*}
		
		We also introduce the space
		\begin{equation*}
			W_L^{s,p}
			=
			\left\{
			v:\mathbb{R}\to\mathbb{R}\; ;\;
			v(t+L)=v(t)
			\ {\rm and} \
			\|v\|_{L^p(0,L)}^{p}
			+
			\int_{0}^{L}\!\!\int_{0}^{L}
			\frac{|v(t)-v(\tau)|^{p}}
			{|t-\tau|^{1+s p}}
			\,\ud t\,\ud\tau
			<\infty
			\right\},
		\end{equation*}
		endowed with the norm
		\begin{equation*}
			\|v\|_{W_L^{s,p}}
			=
			\left(
			\|v\|_{L^p(0,L)}^{p}
			+
			\int_{0}^{L}\!\!\int_{0}^{L}
			\frac{|v(t)-v(\tau)|^{p}}
			{|t-\tau|^{1+s p}}
			\,\ud t\,\ud\tau
			\right)^{1/p}.
		\end{equation*}
		When $p = 2$, this norm is equivalent to
		\begin{equation*}
			\|v\|_{\widetilde{W}_L^{s,p}}
			=
			\left(
			\|v\|_{L^p(0,L)}^{p}
			+
			\int_{0}^{L}\!\!\int_{0}^{L}
			|v(t)-v(\tau)|^{p}
			\mathcal{K}_{s}(t-\tau)\,\ud t\,\ud\tau
			\right)^{1/p},
		\end{equation*}
		where the kernel $\mathcal{K}_{s}$ is defined in \eqref{Km}.
	\end{definition}
    
    Observe that if $u$ is a positive solution to \eqref{eq:frac Hartree} and having the form \eqref{form of u}, where $v(t+L)=v(t)$, then by computation, we have
    \begin{align}\nonumber
       \int_{1<|x|<e^{L}}u(x)(-\Delta)^{s}u(x)\,\ud x
        =&\int_{1}^{e^{L}}\int_{\mathbb{S}^{n-1}}u(r)(-\Delta)^{s}u(r) r^{n-1}\,\ud \theta\,\ud r\\\nonumber
        =&|\mathbb{S}^{n-1}| \int_{0}^{L} e^{(-\frac{n+2s}{2}-\frac{n-2s}{2}+n-1+1)t}v(t)(\mathcal{L}_{s}^{L}v(t))\,\ud t\\\nonumber
        =&|\mathbb{S}^{n-1}|\int_{0}^{L}v (\mathcal{L}_{s}^{L}v(t))\,\ud t\\
        =&|\mathbb{S}^{n-1}|\frac{\kappa_{n,s}}{2}\int_{0}^{L}\!\!\int_{0}^{L}
			\bigl(v(t)-v(\tau)\bigr)^2
			\mathcal{K}_{s}^{L}(t-\tau)\,\ud t\,\ud\tau
			+
			c_{n,s}\int_{0}^{L} v(t)^2\,\ud t.
    \end{align}
    On the other hand, if $u$ satisfies \eqref{assumption of u}, then as discussed in \S\ref{The original fractional Hartree equation}, for every compact set $K\Subset\mathbb R^n\setminus\{0\} $, we obtain that
    \begin{equation*}
        \int_{K}u \left(-\Delta \right)^{s}u\,\ud x <+\infty.
    \end{equation*}
    Combining this with the periodicity assumption, we obtain
	\begin{equation*}
	    \int_{0}^{L}\!\!\int_{0}^{L}
			\bigl(v(t)-v(\tau)\bigr)^2
			\mathcal{K}_{s}^{L}(t-\tau)\,\ud t\,\ud\tau
			+
			\int_{0}^{L} v(t)^2\,\ud t
			<+\infty.
	\end{equation*}

    The space $H_L^s$ is the natural functional setting for equation \eqref{eq:periodic hartree 2}. On the one hand, the operator $\mathcal{L}_s^L$ is well-defined on $H_L^s$. On the other hand, the above argument shows that if a positive solution to \eqref{eq:frac Hartree} satisfying \eqref{assumption of u} admits an $L$-periodic profile $v$ satisfying \eqref{form of u}, then $v\in H_L^s$. Therefore, it is natural to seek periodic solutions of \eqref{eq:periodic hartree 2} in the space $H^{s}_{L}$.
    
    To develop the variational framework, we next give some compactness properties of the space $H_L^{s}$, which will be key in the sequel.
	\begin{proposition}[Compact embedding]\label{compact embedding}
		Let $n \geqslant 2$, $s \in (0,1)$, and $\alpha \in (0,n)$. The embedding
		\[
		H_L^{s} \hookrightarrow L^{q}(0,L)
		\]
		is compact, where
        \begin{equation}\label{compact of q}
          q \in \left(1,\frac{2}{1-2s}\right)
		\quad {\rm if} \quad s \leqslant \tfrac{1}{2},
		\quad {\rm and} \quad
		q \geqslant 1 \quad {\rm if} \quad s > \tfrac{1}{2}.  
        \end{equation}
	\end{proposition}
	\begin{proof}
		According to \cite{DiNezza2012,MR3694655}, the embedding
		\[
		W^{s,2}_{L} \hookrightarrow L^{q}(0,L)
		\]
		is compact, where $q$ satisfies \eqref{compact of q}. Moreover, from the definition of $\mathcal{K}^{L}_{s}$ and the
		positivity of $\mathcal{K}_{s}$, there exists $C>0$ such that
        \begin{equation*}
          \|v\|_{W_L^{s,2}} \leqslant C\, \|v\|_{H_L^{s}}.  
        \end{equation*}
		Combining this with the compact embedding above, we conclude that
        \begin{equation*}
           H_L^{s} \hookrightarrow L^{q}(0,L) 
        \end{equation*}
		is compact.
	\end{proof}
	Thanks to the compact embedding above, the following nonlocal term is well-defined and satisfies the estimate below.
	\begin{proposition}\label{right hand is well-defined}
		Let $n \geqslant 2$, $s \in (0,1)$, and $\alpha \in (0,n)$. There exists a constant $C>0$ such that for every
		$v \in H_L^{s}$,
		\begin{equation}
			\int_0^L \int_0^L
			\mathcal{K}_{-\frac{\alpha}{2}}^{L}(t-\tau)\,
			|v(\tau)|^{p_*}\,|v(t)|^{p_*}\, \ud\tau\,\ud t
			\leqslant C\, \|v\|_{H_L^{s}}^{2p_{\ast}}.
		\end{equation}
		In particular, equation \eqref{eq:periodic hartree 2} is well-defined in
		$H_L^{s}$.
	\end{proposition}
	\begin{proof}
		By Lemma~\ref{Pro of ker}, for fixed $\alpha \in (0,n)$,
		\begin{equation*}
			\mathcal{K}_{-\frac{\alpha}{2}}(t)
			= O\!\left(e^{-\frac{n-\alpha}{2}|t|}\right)
			\quad {\rm as}\ |t| \to \infty,
		\end{equation*}
		which implies that there exist $M_{L,\alpha} \in \mathbb{Z}^{+}$ and a constant $c_{1}>0$ such that
		\begin{equation}\label{MLalpha}
			\sum_{|j|\geqslant M_{L,\alpha}}
			\mathcal{K}_{-\frac{\alpha}{2}}(t+jL)
			\leqslant c_{1}
			\quad {\rm for\ all}\ |t|\leqslant L.
		\end{equation}
		
		If $\alpha \in (0,1)$, in view of the asymptotic behavior of 
		$\mathcal{K}_{-\frac{\alpha}{2}}$ near the origin, there exists
		$0<\delta_{L,\alpha}<L$ such that
		\begin{equation}\label{asymptotic of Kalpha1}
			\mathcal{K}_{-\frac{\alpha}{2}}(t)
			\leqslant c_{2}|t|^{\alpha-1}
			\quad {\rm for\ all}\ |t|\leqslant \delta_{L,\alpha}.
		\end{equation}
		Combining the above estimates with the definition of
		$\mathcal{K}_{-\frac{\alpha}{2}}^{L}$, we obtain
		\begin{align*}
			\int_0^L \int_0^L
			\mathcal{K}_{-\frac{\alpha}{2}}^{L}(t-\tau)\,
			|v(\tau)|^{p_*}\,|v(t)|^{p_*}\, \ud\tau\,\ud t =&
			\int_0^L \int_0^L
			\sum_{|j|\leqslant 1}
			\mathcal{K}_{-\frac{\alpha}{2}}(t-\tau+jL)\,
			|v(\tau)|^{p_*}\,|v(t)|^{p_*}\, \ud\tau\,\ud t \\
			&+
			\int_0^L \int_0^L
			\sum_{1<|j|<M_{L,\alpha}}
			\mathcal{K}_{-\frac{\alpha}{2}}(t-\tau+jL)\,
			|v(\tau)|^{p_*}\,|v(t)|^{p_*}\, \ud\tau\,\ud t \\
			&+
			\int_0^L \int_0^L
			\sum_{|j|\geqslant M_{L,\alpha}}
			\mathcal{K}_{-\frac{\alpha}{2}}(t-\tau+jL)\,
			|v(\tau)|^{p_*}\,|v(t)|^{p_*}\, \ud\tau\,\ud t \\
			\leqslant &
			\, 3c_{2}
			\int_{0}^{L}\int_{0}^{L}
			|t-\tau|^{\alpha-1}
			|v(\tau)|^{p_*}\,|v(t)|^{p_*}\, \ud\tau\,\ud t \\
			&+
			C
			\bigl(
			\|\mathcal{K}_{-\frac{\alpha}{2}}\|_{L^{\infty}
				[\delta_{L,\alpha},\, M_{L,\alpha}L]}
			+ c_{1}
			\bigr) \\
			&\qquad \times
			\int_{0}^{L}\!\!\int_{0}^{L}
			|v(\tau)|^{p_*}\,|v(t)|^{p_*}\, \ud\tau\,\ud t \\
			\leqslant &
			\, C
			\|v\|_{L^{\frac{2p_*}{1+\alpha}}(0,L)}^{2p_*}
			+
			C
			\|v\|_{L^{p_*}(0,L)}^{2p_*} \\
			\leqslant &
			\, C
			\|v\|_{H_L^{s}}^{2p_*},
		\end{align*}
		where the Hardy--Littlewood--Sobolev inequality
		and Proposition~\ref{compact embedding} have been used.
		
		If $\alpha = 1$, the kernel $\mathcal{K}_{-\frac{1}{2}}$
		has a logarithmic singularity at the origin.
		Hence there exist $0<\delta_{L,1}<\min\{1,L\}$ and $c_3>0$
		such that for any $\varepsilon\in(0,1)$,
		\begin{equation}\label{asymptotic of Kalpha2}
			\mathcal K_{-\frac12}(t)
			\leqslant c_3 |\ln |t|| 
			\leqslant \frac{c_3}{\varepsilon}\, |t|^{-\varepsilon},
			\quad |t|\leqslant \delta_{L,1}.
		\end{equation}
		Arguing as in the case $\alpha\in(0,1)$, we obtain
		\begin{equation}
			\int_0^L \int_0^L
			\mathcal K_{-\frac12}^{L}(t-\tau)
			|v(\tau)|^{p_*}|v(t)|^{p_*}\, \ud\tau\,\ud t
			\leqslant
			\frac{3c_{3}}{\varepsilon}
			\int_0^L \int_0^L
			|t-\tau|^{-\varepsilon}
			|v(\tau)|^{p_*}|v(t)|^{p_*}\, \ud\tau\,\ud t
			+ C \|v\|_{L^{p_*}(0,L)}^{2p_*}.
		\end{equation}
		Choose $0<\varepsilon<1$ such that $\varepsilon \varsigma<1$, where
		\[
		\varsigma=\frac{2n}{n-1}.
		\]
		Then $|t|^{-\varepsilon}\in L^\varsigma(0,L)$, and by H\"older's inequality
		\begin{equation}
			\int_0^L \int_0^L
			|t-\tau|^{-\varepsilon}
			|v(\tau)|^{p_*}|v(t)|^{p_*}\, \ud\tau\,\ud t
			\leqslant
			C \|v\|_{L^{p_*}(0,L)}^{p_*}
			\|v\|_{L^{p_{\ast}\varsigma'}(0,L)}^{p_*}.
		\end{equation}
		Since $p_{\ast}\varsigma' = \frac{2n}{n-2s} < \frac{2}{1-2s}$, the compact embedding in Proposition~\ref{compact embedding} gives
		\[
		\int_0^L \int_0^L
		\mathcal K_{-\frac12}^{L}(t-\tau)
		|v(\tau)|^{p_*}|v(t)|^{p_*}\, \ud\tau\,\ud t
		\leqslant C \|v\|_{H_L^s}^{2p_*}.
		\]
		
		If $\alpha \in (1,n)$, the boundedness of the kernel $\mathcal{K}_{-\frac{\alpha}{2}}$ together with \eqref{MLalpha} implies that
		\begin{equation}
			\int_0^L \int_0^L
			\mathcal{K}_{-\frac{\alpha}{2}}^{L}(t-\tau)\,
			|v(\tau)|^{p_*}\,|v(t)|^{p_*}\, \ud\tau\,\ud t  \leqslant C \|v\|_{L^{p_{\ast}}(0,L)}^{2p_{\ast}}\leqslant C \|v\|_{H_L^{s}}^{2p_{\ast}}
		\end{equation}
	\end{proof}
	Within the framework of the function space defined above, we state the following strong maximum principle; its proof can be found in \cite{MR3694655}.
	\begin{proposition}[Strong maximum principle]\label{maximum principle}
		Let $n \geqslant 2$, $s \in (0,1)$, and $\alpha \in (0,n)$. Let $v \in H_{L}^{s} \cap \mathcal{C}^{0}(\mathbb{R})$ with $v \geqslant 0$ be a solution to
		\begin{equation*}
			\mathcal{L}_{s}v=f(v), \quad{\rm in}\ \mathbb{R}
		\end{equation*}
		where $f$ satisfies $f(v)\geqslant 0$ if $v \geqslant 0$. Then $v>0$ or $v \equiv0$. Here, $f(v)$ denotes a general nonnegative right-hand side; in our application, $f$ is the nonlocal term \eqref{f}.
	\end{proposition}
	\subsection{$L^\infty$ regularity}
	In the following Proposition~\ref{L infty for s<1/2}, we prove the $L^\infty$ regularity,
	using the equivalent characterization of $\mathcal{L}_s$ as a Dirichlet-to-Neumann operator
	for problem \eqref{eq:periodic extension-normalized}. First, we fix some notation that will be used below. For $0<R<\rho_{0}^{\ast}$, we denote
	\begin{align*}
		B_R^{+}
		&=
		\left\{ (t,\rho^{*}) \in \mathbb{R}^{2} : \rho^{*} > 0,\; |(t,\rho^{*})| < R \right\}, \\
		\Gamma_R^{0}
		&=
		\left\{ (t,0) \in \partial \mathbb{R}^{2}_{+} : |t| < R \right\}.
	\end{align*}

	\begin{proposition}\label{L infty for s<1/2}
		Let $n \geqslant 2$, $s \in (0,1/2)$, and $\alpha \in (\alpha_{\ast},n)$. Let $V=V(t,\rho^{\ast})$ be a nonnegative weak solution to the problem
		\begin{equation}\label{eq:prop3.4-extension}
			\begin{cases}
				-\operatorname{div}_{g^*}\!\big((\rho^*)^{1-2s}
				\nabla_{g^*} V\big)=0
				& {\rm in}\ (B_{R}^{+},g^{\ast}),\\[6pt]

				- d_s \displaystyle\lim_{\rho^*\to0}
				(\rho^*)^{1-2s}\partial_{\rho^*}V
				+ c_{n,s}\,v
				=
				B(t)v(t)^{p_{\ast}-1}
				& {\rm on}\ \Gamma_{R}^{0},
			\end{cases}
		\end{equation}
		where $v(t) = V(t,0)$ is $L$-periodic, $R>L$ is fixed, $p_* > 1$, and
    \begin{equation}\label{eq:degiorgi-B}
        B(t) := \int_0^L \mathcal{K}_{-\frac{\alpha}{2}}^L(t-\tau)\,
	v(\tau)^{p_*} \, \ud\tau.
    \end{equation}
	If
	\begin{equation}\label{degiorgi-alpha choice}
	\int_{0}^{L}|v|^{\frac{2}{1-2s}} \, \ud t := \zeta<+\infty,
	\end{equation}
	then $v \in L^\infty(0,L)$, and there exists a constant $C > 0$ depending on $n, s, \alpha, L$, and $\zeta$ such that
	\begin{equation}\label{eq:degiorgi-conclusion}
		\|v\|_{L^\infty(0,L)} \leqslant C.
	\end{equation}
	\end{proposition}
  
	\begin{proof}
    The proof is based on the classical De Giorgi truncation method. It is carried out within the extension framework of Cabr\'e and Sire~\cite{CabreSire2014}, relying on the degenerate elliptic theory for $A_2$ weights developed by Fabes, Kenig, and Serapioni~\cite{FabesKenigSerapioni1982}. For the classical De Giorgi method, we refer to Giusti~\cite[Chapter~7]{Giusti2003}; see also Di Castro, Kuusi, and Palatucci~\cite{DiCastroKuusiPalatucci2014} for related developments in the nonlocal setting.
    
		After flattening the boundary and using local coordinates,
		problem \eqref{eq:prop3.4-extension} is locally equivalent to the Euclidean
		weighted extension problem (the curvature terms arising from the metric $g^*$ produce lower-order perturbations near $\rho^* = 0$ that are absorbed into the iteration; cf.\ \cite[Proposition~3.4]{MR3694655})
		\begin{equation}\label{eq:degiorgi-extension}
		\begin{cases}
			-\operatorname{div}(y^{a}\nabla V)=0
			& {\rm in}\ B_{R}^{+},\\[6pt]
			- y^{a}\partial_y V
			+ c_{n,s}\,v
			=
			B(t)\, v^{p_{\ast}-1}
			& {\rm on}\ \Gamma_{R}^{0},
		\end{cases}
	\end{equation}
    where $a:=1-2s$.
	
    We divide the proof into several claims.

\medskip

\noindent{\bf Claim~1} {\rm (Truncation setup):} {\it The test function $W_k$ satisfies}
\begin{align}\label{eq:degiorgi-weak}
	\int_{B_{R}^+} y^a \left| \nabla W_k \right|^{2}\, \ud t \, \ud y
	&= \int_{A_k} \left( B(t) v^{p_*-1} - c_{n,s} v \right) w_k \, \ud t,
\end{align}

\noindent{\it Proof.}
For $k \geqslant 0$, we define the truncations
\[
w_k := (v - k)_+ = \max\{v - k, 0\}, \qquad W_k := (V - k)_+ = \max\{V - k, 0\},
\]
and the superlevel set
\[
A_k := \{(t,0) \in \Gamma_{R}^{0} : v(t) > k\}.
\]
Since $V$ is a nonnegative weak solution to \eqref{eq:degiorgi-extension}, we test the equation with $W_k$. This is an admissible test function: indeed, $W_k = (V-k)_+ \in W^{1,2}(y^a, B_{R}^+)$ since truncation preserves weighted Sobolev regularity. On $\Gamma_{R}^0$, the trace of $W_k$ equals $w_k$, which is supported on $A_k$. Testing the weak formulation of \eqref{eq:degiorgi-extension} with $W_k$ and using that $\nabla W_k = \nabla V$ on $\{V > k\}$ and $W_k = 0$ on $\{V \leqslant k\}$, we obtain \eqref{eq:degiorgi-weak}.

\medskip

\noindent{\bf Claim~2} {\rm (Nonlinear bound):} For any $k > 0$, the nonlinear term satisfies
\begin{equation}\label{eq:degiorgi-nonlinear-split}
	v^{p_*-1} w_k \leqslant C_p \left( w_k^{p_*} + k^{p_*-1} w_k \right) \quad {\rm on} \quad A_k.
\end{equation}

\noindent{\it Proof.}
On the set $A_k\subset \Gamma_{R}^{0}$ with $k > 0$, we write $v = w_k + k$ and use the elementary inequality
\[
(a + b)^{p_*-1} \leqslant 2^{(p_*-2)_+} \left( a^{p_*-1} + b^{p_*-1} \right), \qquad a,b \geqslant 0,
\]
which holds for all $p_* > 1$. We stress that no restriction $p_* \geqslant 2$ is needed here, since $v > k > 0$ on $A_k$ and there is no singularity at $v = 0$. Applying this with $a = w_k$ and $b = k$ and multiplying both sides by $w_k \geqslant 0$, we obtain
\[
v^{p_*-1} w_k = (w_k + k)^{p_*-1} w_k \leqslant 2^{(p_*-2)_+}\left( w_k^{p_*} + k^{p_*-1} w_k \right),
\]
which is \eqref{eq:degiorgi-nonlinear-split} with $C_p = 2^{(p_*-2)_+}$.

\medskip

Combining \eqref{eq:degiorgi-nonlinear-split} with the fact that $-c_{n,s}\int_{A_k} v\,w_k \leqslant 0$ on the right-hand side of \eqref{eq:degiorgi-weak}, we obtain
\begin{equation}\label{eq:degiorgi-weak with I1 I2}
    \int_{B_{R}^+} y^a \left| \nabla W_k \right|^{2}\, \ud t \, \ud y
	\leqslant C \int_{A_k}   B(t) w_{k}^{p_*} \, \ud t+C\int_{A_k}B(t)k^{p_{\ast}-1}w_{k}  \,\ud t =: C(I_{1}+k^{p_{\ast}-1}I_{2}).
\end{equation}

\medskip
\noindent{\bf Claim~3} {\rm (Estimate of $I_{1}$ and $I_{2}$):} For any $q_1'$ and $q_2'$ satisfying the conditions specified below,
\begin{equation}\label{eq:degiorgi I1 estimate}
	I_1
\leqslant
C
|A_k|^{\frac{1}{q_1'}-\frac{p_*(1-2s)}{2}}
\|w_k\|_{L^{\frac{2}{1-2s}}(A_{k})}^{p_*},
\end{equation}
and
\begin{equation}\label{eq:degiorgi I2 estimate}
	I_2
\leqslant
C
|A_k|^{\frac1{q_2'}-\frac{1-2s}{2}}
\|w_k\|_{L^{\frac2{1-2s}}(A_k)}.
\end{equation}

\noindent{\it Proof.}
We first estimate $I_{1}$. By H\"older's inequality,
\begin{equation}\label{eq:degiorgi-I1 holder}
    I_1
\leqslant
C
\|B\|_{L^{q_1}(0,L)}
\|w_k^{p_*}\|_{L^{q_1'}(A_k)},
\end{equation}
where $q_1>1$ and $q_1'>1$ satisfy $
\frac1{q_1}+\frac1{q_1'}=1$, and $C$ depends only on $L$ and $R$. Since $B$ is defined by \eqref{eq:degiorgi-B}, Young's convolution inequality and Proposition~\ref{compact embedding} yield
\begin{equation*}
    \|B\|_{L^{q_1}(0,L)}
\leqslant
\|\mathcal{K}_{-\frac{\alpha}{2}}^{L}\|_{L^{r}(0,L)}
\|v^{p_*}\|_{L^{p}(0,L)}
\end{equation*}
provided that $\frac1{q_1}
=
\frac1r+\frac1p-1$ and $pp_{*}\leqslant
\frac{2}{1-2s}$.

If $0<\alpha<1$, then by Lemma~\ref{Pro of ker},
\begin{equation*}
    \mathcal{K}_{-\frac{\alpha}{2}}^{L}\in L^{r}(0,L),
\quad
1\leqslant r<\frac1{1-\alpha}.
\end{equation*}
Combined with $\frac1{q_1'}
=
2-\frac1r-\frac1p$ and $pp_{*} \leqslant \frac{2}{1-2s}$, we have
\begin{equation*}
    q_1'>
\frac1{1+\alpha-\frac{p_*(1-2s)}2}.
\end{equation*}
On the other hand, we assume that $p_*q_1'
\leqslant
\frac{2}{1-2s}$, or equivalently, $q_1'
\leqslant
\frac{2}{p_*(1-2s)}$, so that H\"older's inequality can be applied in the estimate of $\|w_k\|_{L^{p_*q_1'}(A_k)}$. Therefore, throughout the following argument, we assume that
\begin{equation*}
    \max\left\{
1,
\frac1{1+\alpha-\frac{p_*(1-2s)}2}
\right\}
<
q_1'
\leqslant
\frac2{p_*(1-2s)},
\end{equation*}
where the nonemptiness of the above interval is verified in Claim~4. 

If $1\leqslant\alpha<n$, then by Lemma~\ref{Pro of ker},
\begin{equation*}
    \mathcal{K}_{-\frac{\alpha}{2}}^{L}\in L^{r}(0,L)
\end{equation*}
for every $1\leqslant r<\infty$ (and $r=\infty$ if $\alpha>1$). Hence, for any $q_{1}'$ satisfying
\[
1<q_1'
\leqslant
\frac2{p_*(1-2s)},
\]
the estimate \eqref{eq:degiorgi-I1 holder} holds.

Consequently,
\begin{equation*}
    \|B\|_{L^{q_1}(0,L)}
\leqslant
C.
\end{equation*}
Since $p_*q_1'
\leqslant
\frac2{1-2s}$, H\"older's inequality implies
\begin{equation*}
    \|w_k^{p_*}\|_{L^{q_1'}(A_k)}
\leqslant
|A_k|^{\frac1{q_1'}-\frac{p_*(1-2s)}2}
\|w_k\|_{L^{\frac2{1-2s}}(A_k)}^{p_*},
\end{equation*}
and hence 
\begin{equation*}
    I_1
\leqslant
C
|A_k|^{\frac1{q_1'}-\frac{p_*(1-2s)}2}
\|w_k\|_{L^{\frac2{1-2s}}(A_k)}^{p_*}
\end{equation*}
for any $q_{1}'$ satisfying
\begin{equation}\label{degiorgi-q_{1}'in claim 3}
  \max\left\{
1,
\frac1{1+\alpha-\frac{p_*(1-2s)}2}
\right\}
<
q_1'
\leqslant
\frac2{p_{\ast}(1-2s)}.
\end{equation}
Indeed, if $\alpha\geqslant1$, then $1+\alpha-\frac{p_*(1-2s)}2>1$, and therefore $\max\left\{
1,
\frac{1}{1+\alpha-\frac{p_*(1-2s)}2}
\right\}=1$.

We next estimate $I_2$. Arguing as in the estimate of $I_1$, we conclude that
the estimate \eqref{eq:degiorgi I2 estimate}
holds for any $q_{2}'$ satisfying
\begin{equation}\label{degiorgi-q_{2}'in claim 3}
   \max\left\{
1,
\frac1{1+\alpha-\frac{p_*(1-2s)}2}
\right\}
<
q_2'
\leqslant
\frac2{1-2s}.
\end{equation}

\medskip

For $k \geqslant 0$, set
\begin{equation*}
    \psi(k) := \left\|w_k\right\|_{L^{2/(1-2s)}(A_{k})}^{2/(1-2s)} \quad {\rm and} \quad \varphi(k) := |A_k|.
\end{equation*}
\noindent{\bf Claim~4} {\rm (De Giorgi iteration):} {\it For fixed $k_{0}>0$, there exists $d > 0$ such that $\psi(k_0 + d) = 0$ and the $L^\infty$-estimate holds}
\[
\|v\|_{L^\infty(0,L)} \leqslant k_0 + d.
\]

\noindent{\it Proof.} The trace Sobolev embedding in \cite[Corollary 5.3]{MR3148060} gives
\[
\psi(k)^{1-2s}=\left\| w_k\right\|_{L^{2/(1-2s)}(A_{k})}^2 \leqslant C_S \int_{B_{R}^+} y^a |\nabla W_k|^2 \, \ud t \, \ud y,
\]
which combined with \eqref{eq:degiorgi-weak with I1 I2} and Claim 3 yields
\begin{equation}
    \psi(k)^{1-2s} \leqslant C\varphi(k)^{\frac{1}{q_1'}-\frac{p_*(1-2s)}{2}}
\psi(k)^{\frac{p_*(1-2s)}{2}}+Ck^{p_{*}-1}\varphi(k)^{\frac1{q_2'}-\frac{1-2s}{2}}
\psi(k)^{\frac{1-2s}{2}}
\end{equation}
for any $q_{1}'$ and $q_{2}'$ satisfying \eqref{degiorgi-q_{1}'in claim 3} and \eqref{degiorgi-q_{2}'in claim 3}, respectively. By Chebyshev's inequality, one has
\begin{equation*}
\varphi(h) \leqslant \frac{1}{(h-k)^{2/(1-2s)}} \psi(k) \quad {\rm for\ all}\ h > k.
\end{equation*}
Hence, for $h>k$, we have
\begin{equation}\label{eq:degiorgi-psik^1-2s}
    \psi(h)^{1-2s} \leqslant C\frac{\psi(k)^{\frac{1}{q_{1}'}}}{(h-k)^{\gamma_{1}}}+C\frac{h^{p_{*}-1}\psi(k)^{\frac{1}{q_{2}'}}}{(h-k)^{\gamma_{2}}},
\end{equation}
where 
\begin{equation}
    \gamma_{1}=\left(\frac{2}{1-2s}\right)\left(\frac{1}{q_{1}'}-\frac{p_{*}(1-2s)}{2}\right)  \quad \text{ and } \quad \gamma_{2}=\left(\frac{2}{1-2s}\right)\left(\frac{1}{q_{2}'}-\frac{1-2s}{2}\right).
\end{equation}

Next, we choose $q_1'$ and $q_2'$ satisfying 
\begin{equation}\label{degiorgi-q_{1}'in claim 4}
   \max\left\{
1,
\frac1{1+\alpha-\frac{p_*(1-2s)}2}
\right\}
<
q_1'=q_{2}'
<
\min \left\{\frac1{1-2s},\frac{2}{p_{*}(1-2s)} \right\}.
\end{equation}
Such a choice is possible under the assumption \eqref{degiorgi-alpha choice} on $\alpha$, and the above conditions are stronger than those required in Claim~3. Indeed, if $n-4s \leqslant \alpha<n$ (equivalently, $p_{*} \geqslant 2$), then $\min \left\{\frac1{1-2s},\frac{2}{p_{*}(1-2s)} \right\}=\frac{2}{p_{*}(1-2s)}$. Since $n\geqslant2$ and $0<s<\frac12$, $\frac2{p_*(1-2s)}>1$. Moreover,
\begin{equation*}
    \frac1{1+\alpha-\frac{p_*(1-2s)}2}
<
\frac2{p_*(1-2s)}.
\end{equation*}
Therefore,
\begin{equation}\label{degiorgi-nonempty as alpha>n-4s}
    \max\left\{
1,
\frac1{1+\alpha-\frac{p_*(1-2s)}2}
\right\}
<
\frac2{p_*(1-2s)}, \quad \text{ if } \quad n-4s\leqslant\alpha<n.
\end{equation}
Now we consider $\max\left\{
0,\,
\alpha_{\ast}\right\}<\alpha<n-4s$. First, we observe that $\alpha_{\ast}<n-4s$. Indeed, on the one hand
\begin{equation*}
    (n-4s)-\alpha_{\ast}=
\frac{n(2n-3-2s)+4s}{2n-1-2s}>0,
\end{equation*}
since $n\geqslant2$ and $0<s<\frac12$. Since $\alpha<n-4s$, we have $1<p_{*}<2$, which implies that 
\[
\min \left\{\frac1{1-2s},\frac{2}{p_{*}(1-2s)} \right\}=\frac{1}{1-2s}.
\]
On the other hand, by the assumption, it holds 
\(\alpha>
\alpha_{\ast}\), from which we obtain 
\[
\frac1{1+\alpha-\frac{p_*(1-2s)}2}
<
\frac1{1-2s}.
\]
Hence, we get
\begin{equation}\label{degiorgi-nonempty as alpha<n-4s}
    \max\left\{
1,
\frac1{1+\alpha-\frac{p_*(1-2s)}2}
\right\}
<
\frac1{1-2s}\quad \text{ if } \quad \max\left\{
0,\,
\alpha_{\ast}
\right\}<\alpha<n-4s.
\end{equation}
By \eqref{degiorgi-nonempty as alpha>n-4s} and \eqref{degiorgi-nonempty as alpha<n-4s}, we get that the interval in \eqref{degiorgi-q_{1}'in claim 4} is nonempty.

For simplicity, we set $q':=q_1'=q_2'$. Hence, we can rewrite \eqref{eq:degiorgi-psik^1-2s} as
\begin{equation}\label{eq:degiorgi-psik}
    \psi(h)\leqslant C\left(\frac{1}{(h-k)^{\gamma_{1}}}+\frac{h^{p_{*}-1}}{(h-k)^{\gamma_{2}}}\right)^{\frac{1}{1-2s}}\psi(k)^{\beta},
\end{equation}
where $\beta=\frac{1}{q'(1-2s)}$, which satisfies $\beta>1$ by \eqref{degiorgi-q_{1}'in claim 4}.

We are now in a position to perform the De Giorgi iteration. Fix $k_0>0$, and let $d>k_{0}$ be determined later. Let us define
\begin{equation*}
    k_j:=k_0+d\left(1-2^{-j}\right),\quad {\rm for} \quad j\in\mathbb{N}_0
\end{equation*}
and $Y_{j}=\psi(k_{j})$.
Hence, \eqref{eq:degiorgi-psik} can be rewritten as
\begin{align*}
    Y_{j+1} &\leqslant C\left(2^{(j+1)\gamma_{1}}d^{-\gamma_{1}}+(2d)^{p_{*}-1}2^{(j+1)\gamma_{2}}d^{-\gamma_{2}}\right)^{\frac{1}{1-2s}}Y_{j}^{\beta} \\
    &\leqslant C 2^{(j+1)\frac{\gamma_{2}}{1-2s}}\,d^{\frac{-\gamma_{1}}{1-2s}}Y_{j}^{\beta},
\end{align*}
where we have used $0<\gamma_1<\gamma_2$, $-\gamma_1=p_*-1-\gamma_2$, and absorbed the factor $2^{p_*-1}$ into the constant $C$. Iterating the above inequality, we obtain
\begin{equation*}
    Y_j
\leqslant
\left(Cd^{-\frac{\gamma_1}{1-2s}}\right)^{\frac{\beta^j-1}{\beta-1}}
2^{\frac{\gamma_2}{1-2s}\cdot
\frac{\beta^{j+1}-(j+1)\beta+j}{(\beta-1)^2}}
Y_0^{\beta^j} \leqslant \left(C^{\frac{1}{\beta-1}}d^{-\frac{\gamma_{1}}{(1-2s)(\beta-1)}} 2^{\frac{\gamma_{2}\beta}{(1-2s)(\beta-1)^{2}}}Y_{0} \right)^{\beta^{j}}.
\end{equation*}
Choose $d>k_0$ sufficiently large such that
\begin{equation*}
    C^{\frac1{\beta-1}}
d^{-\frac{\gamma_1}{(1-2s)(\beta-1)}}
2^{\frac{\gamma_2\beta}{(1-2s)(\beta-1)^2}}
Y_0<1.
\end{equation*}
Then $Y_j\longrightarrow0$ as $j\to\infty$.  Since $k_j\uparrow k_0+d$, it follows that $|A_{k_0+d}|=0$. Hence,
\begin{equation*}
    v(t)\leqslant k_0+d\quad\text{for a.e. }t\in(0,L),
\end{equation*}
which yields $\|v\|_{L^\infty(0,L)}\leqslant k_0+d$. 

\medskip

This completes the proof of Proposition~\ref{L infty for s<1/2}.

	\end{proof}

      \begin{remark}\label{rem:lower-bound-vanishes}
      For $0<s<\frac12$, the lower bound $\alpha_{\ast}$ is always strictly less than $1$. Moreover, if 
\begin{equation*}
    \frac{3n-\sqrt{9n^2-8n}}8\leqslant s<\frac12,
\end{equation*}
then this lower bound is nonpositive. Consequently, $\alpha_{\ast}=0$. Hence, the assumption \eqref{degiorgi-alpha choice} in Proposition~\ref{L infty for s<1/2} reduces simply to
\begin{equation*}
    0<\alpha<n.
\end{equation*}
\end{remark}

	\begin{remark}\label{L infty for 0<s<1}
    For $ \frac{1}{2}\leqslant s<1$, the $L^\infty$ regularity holds for every $0<\alpha<n$. When $s=\frac12$, the conclusion can be established by a similar argument. In this case, the condition $\alpha_{\ast}<\alpha<n$ reduces simply to $0<\alpha<n$. When $s > 1/2$, the one-dimensional Sobolev embedding $H^s(0,L) \hookrightarrow L^\infty(0,L)$ applies directly.
	\end{remark}
	\subsection{Higher regularity of solutions}
	In this subsection, we bootstrap the $L^\infty$ bound obtained above to full smoothness. We first establish H\"older continuity of $v$ by applying the nonlocal Harnack inequality, and then prove that the convolution term $B(t)$ inherits the regularity of $v$, which in turn yields higher Sobolev regularity by iteration.

	For convenience, we set 
	\begin{equation}\label{f}
		f:=\left(\int_{-\infty}^{+\infty}\mathcal{K}_{-\frac{\alpha}{2}}(t-\tau)v(\tau)^{p_{\ast}} \ud\tau\right)v(t)^{p_{\ast}-1}.
	\end{equation}

	We first study the H\"older continuity of $v$. 
	To this end, we apply the regularity result in \cite[Proposition 3.8]{MR3694655}. 
	Since our kernel corresponds to a tempered stable process, the proof of this regularity result also follows the ideas in Kassmann \cite{Kassmann2009} and Silvestre \cite{Silvestre2006} concerning H\"older continuity.

	\begin{proposition}\label{Cbeta estimate}
		Let $n \geqslant 2$, $s \in (0,1)$, and $\alpha \in (0,n)$. If $v \in H^{s}_{L} \cap L^{\infty}(\mathbb{R})$ is a solution to \eqref{eq:Lgamma=hartree}, then $f \in L^{q}(0,L)$ for every $q \geqslant 1$. Moreover, there exist constants $c > 0$ and $\beta \in (0,1)$ such that, for any $R \in (0,1)$,
		\begin{equation*}
			|v(t) - v(\tau)|
			\leqslant
			c\,|t - \tau|^{\beta}
			\left(
			R^{-\beta}\|v\|_{L^\infty(\mathbb R)}
			+
			\|f\|_{L^q(0,L)}
			\right).
		\end{equation*}
	\end{proposition}
	
	\begin{proof}
		Since $v\in H_L^s$, it is $L$-periodic. Together with $v\in L^\infty(\mathbb{R})$, we have $v\in L^p(0,L)$ for every $1\leqslant p<\infty$.
		
		Let $q \geqslant 1$ be fixed. Then
		\begin{equation*}
			\|f\|_{L^q(0,L)}^q
			=
			\int_0^L A(t)B(t)\,\ud t,
		\end{equation*}
		where
		\begin{equation*}
			A(t)=|v(t)|^{(p_{\ast}-1)q}, \qquad
			B(t)=\left|\int_{0}^{L}\mathcal{K}_{-\frac{\alpha}{2}}^{L}(t-\tau)\, v(\tau)^{p_{\ast}}\, \ud\tau\right|^q.
		\end{equation*}
		Applying H\"older's inequality with exponents $r,r'>1$ satisfying
		$\frac{1}{r}+\frac{1}{r'}=1$, we obtain
		\begin{equation*}
			\int_0^L A(t)B(t)\,\ud t
			\leqslant
			\Bigl(\int_0^L A(t)^r\,\ud t\Bigr)^{1/r}
			\Bigl(\int_0^L B(t)^{r'}\,\ud t\Bigr)^{1/r'}.
		\end{equation*}
		We have $A(t)^r = |v(t)|^{(p_\ast-1)qr}$. Choosing $r$ such that
		\begin{equation*}
			(p_\ast-1)qr = p,
		\end{equation*}
		it follows that
		\begin{equation*}
			\int_0^L A(t)^r\,\ud t
			= \int_0^L |v(t)|^{p}\,\ud t < \infty.
		\end{equation*}
		Since $\mathcal{K}_{-\frac{\alpha}{2}}^{L} \in L^1(0,L)$,
		Young's inequality on $(0,L)$ yields
		\begin{equation*}
			\Bigl\|
			\int_{0}^{L}\mathcal{K}_{-\frac{\alpha}{2}}^{L}(\cdot-\tau)
			\, v(\tau)^{p_{\ast}}\, \ud\tau
			\Bigr\|_{L^{qr'}(0,L)}
			\leqslant
			\|\mathcal{K}_{-\frac{\alpha}{2}}^{L}\|_{L^1(0,L)}
			\, \|v^{p_{\ast}}\|_{L^{qr'}(0,L)}.
		\end{equation*}
		Choosing $r'>1$ such that
		\begin{equation*}
			p_{\ast} q r' = p,
		\end{equation*}
		we obtain
		\begin{equation*}
			\int_0^L B(t)^{r'}\,\ud t < \infty.
		\end{equation*}
		The condition $\frac{1}{r}+\frac{1}{r'}=1$ then gives
		\begin{equation*}
			q = \frac{p}{2p_\ast-1}.
		\end{equation*}
		Taking $p\gg1$ sufficiently large, we obtain $q>n$. Hence
		\begin{equation*}
			f \in L^q(0,L) \quad {\rm for\ some}\ q>n.
		\end{equation*}
		
		We rewrite the equation as
		\begin{equation*}
			\mathcal L_s v = f \quad {\rm in} \quad B_R(x_0).
		\end{equation*}
		Since $v$ is $L$-periodic, $f \in L^q(0,L)$ implies
		$f \in L^q_{\loc}(\mathbb R)$.
		Moreover, any ball $B_R(x_0)$ can be covered by finitely many translates of $(0,L)$, so that
		\begin{equation*}
			\|f\|_{L^q(B_R(x_0))}
			\leqslant C\,\|f\|_{L^q(0,L)}.
		\end{equation*}
		Therefore, by \cite[Proposition 3.8]{MR3694655}, we obtain
		\begin{equation*}
			|v(t)-v(\tau)|
			\leqslant
			c |t-\tau|^\beta
			\Big(
			R^{-\beta}\|v\|_{L^\infty(\mathbb R)}
			+
			\|f\|_{L^q(0,L)}
			\Big).
		\end{equation*}
	\end{proof}
	Now, we recall the following regularity improvement result from \cite[Proposition~3.9]{MR3694655}, whose proof relies on the H\"older estimate for the fractional Laplacian due to Dong and Kim \cite{DongKim2013} together with the Schauder regularity theory of Silvestre \cite{Silvestre2006}.
	\begin{proposition}\label{regularity improvement}
		Let $n \geqslant 2$, $s \in (0,1)$, $\alpha \in (0,n)$, and $\beta \in (0,1)$. Assume that $f \in \mathcal{C}^{\beta}(\mathbb{R})$ and let
		$v \in L^{\infty}(\mathbb{R}) \cap \mathcal{C}^{\beta} (\mathbb{R})$ be a solution to \eqref{eq:Lgamma=hartree}.
		Then there exists $c > 0$, depending on $n$, $\beta$, and $s$, such that
		\[
		\|v\|_{\mathcal{C}^{\beta+2s}(\mathbb{R})}
		\leqslant
		c\bigl(\|v\|_{\mathcal{C}^{\beta}(\mathbb{R})} + \|f\|_{\mathcal{C}^{\beta}(\mathbb{R})}\bigr).
		\]
	\end{proposition}
	
	However, the above regularity improvement alone is not sufficient for our purposes. When the right-hand side is a local nonlinearity, such as $v^p$ with $p>1$, the regularity of $v$ immediately yields that of the nonlinearity. 
	In contrast, in our case the right-hand side is a nonlocal term $f$ defined in \eqref{f}, whose regularity does not follow directly from that of $v$. Therefore, it is necessary to analyze the regularity of $f$ in order to apply the above result and obtain the desired regularity improvement.
	\begin{proposition}\label{f regularity}
		Let $n \geqslant 2$, $s \in (0,1)$, and $\alpha \in (0,n)$. If $v \in H_{L}^{s} \cap \mathcal{C}^{k,\ell}(\mathbb{R})$ is a positive solution to \eqref{eq:Lgamma=hartree} for some $k \in \mathbb{N}$ and $0<\ell<1$, then $f \in \mathcal{C}^{k,\ell}(\mathbb{R})$.
	\end{proposition}
	\begin{proof}
		Since $v \in \mathcal{C}^{k,\ell}(\mathbb{R})$ and $v(x)>0$ for all $x \in \mathbb{R}$,
		we have $v^{p} \in \mathcal{C}^{k,\ell}(\mathbb{R})$ for any real $p$,
		in particular for $p=p_\ast-1$ and $p=p_\ast$.
		Therefore, to prove that $f \in \mathcal{C}^{k,\ell}(\mathbb{R})$,
		it suffices to show that $B \in \mathcal{C}^{k,\ell}(\mathbb{R})$, where
		\begin{equation}\label{B(t)}
			B(t):=\int_{\mathbb{R}}
			\mathcal{K}_{-\frac{\alpha}{2}}(t-\tau)\,v(\tau)^{p_{\ast}}
			\, \ud\tau
			=\int_{\mathbb{R}}
			\mathcal{K}_{-\frac{\alpha}{2}}(\tau)\,v(t-\tau)^{p_{\ast}}
			\, \ud\tau .
		\end{equation}
		
		We first consider the case $k=0$.
		For any $t_{1}, t_{2} \in \mathbb{R}$, taking absolute values and using that
		$v^{p_{\ast}} \in \mathcal{C}^{\ell}(\mathbb{R})$, we obtain
		\begin{align*}
			|B(t_{1})-B(t_{2})|
			&\leqslant \int_{\mathbb{R}}\mathcal{K}_{-\frac{\alpha}{2}}(\tau)
			\bigl|v(t_{1}-\tau)^{p_{\ast}}
			- v(t_{2}-\tau)^{p_{\ast}}\bigr|
			\, \ud\tau \\
			&\leqslant C |t_{1}-t_{2}|^{\ell}
			\int_{\mathbb{R}}\mathcal{K}_{-\frac{\alpha}{2}}(\tau)
			\, \ud\tau \\
			&= C \|\mathcal{K}_{-\frac{\alpha}{2}}\|_{L^{1}(\mathbb{R})}
			|t_{1}-t_{2}|^{\ell},
		\end{align*}
		which shows that $B \in \mathcal{C}^{0,\ell}(\mathbb{R})$.
		
		For $k \geqslant 1$, we differentiate with respect to $t$.
		Since $v^{p_\ast} \in \mathcal{C}^{k,\ell}(\mathbb{R})$ and
		$\mathcal{K}_{-\frac{\alpha}{2}} \in L^1(\mathbb{R})$,
		differentiation can be passed under the integral sign, yielding
		\[
		\partial_t^j B(t)
		= \int_{\mathbb{R}}
		\mathcal{K}_{-\frac{\alpha}{2}}(\tau)\,
		\partial_t^j\!\bigl(v(t-\tau)^{p_\ast}\bigr)
		\, \ud\tau,
		\quad j=1,\dots,k.
		\]
		Since $v^{p_\ast} \in \mathcal{C}^{k,\ell}(\mathbb{R})$, it follows that
		$\partial_t^j (v^{p_\ast}) \in \mathcal{C}^{0,\ell}(\mathbb{R})$ for all $j \leqslant k$.
		Therefore, the estimate for the case $k=0$ applies to each
		$\partial_t^j B$, and we obtain
		$\partial_t^j B \in \mathcal{C}^{0,\ell}(\mathbb{R})$ for $j=1,\dots,k$.
		Consequently, $B \in \mathcal{C}^{k,\ell}(\mathbb{R})$.
	\end{proof}
	\begin{remark}\label{remark C infty}
		Proposition~\ref{Cbeta estimate} implies that any solution
		$v \in H_{L}^{s} \cap L^{\infty}(\mathbb{R})$ of equation \eqref{eq:Lgamma=hartree}
		belongs to $\mathcal{C}^{\beta}(\mathbb{R})$ for some $\beta \in (0,1)$. Moreover, if $v$ is positive, then Propositions~\ref{regularity improvement}
		and~\ref{f regularity} allow one to bootstrap the regularity and conclude that
		$v \in \mathcal{C}^{\infty}(\mathbb{R})$.
	\end{remark}
	\section{Minimization problem}\label{sec:minimization}
	Building on the variational framework and regularity results established in the previous section, we now prove the existence of Delaunay-type solutions by establishing the existence of minimizers and completing the proof of Theorem~\ref{thm:main}.

	\subsection{Variational formulation}
	For fixed $\alpha \in (0,n)$, we consider the following minimization problem
	\begin{equation}\label{eq:minimization}
		c(L)=\inf_{v \in H_{L}^{s},\,v\not\equiv0}\mathscr{F}_{L}(v),
	\end{equation}
	where the Rayleigh-type quotient $\mathscr{F}_{L}$ is defined by
	\begin{equation}\label{eq:Rayleigh}
		\mathscr{F}_{L}(v)=\frac{\frac{\kappa_{n,s}}{2}\int_{0}^{L}\int_{0}^{L}\left(v(t)-v(\tau)\right)^{2}\mathcal{K}_{s}^{L}(t-\tau)\,\ud \tau\, \ud t+c_{n, s} \int_{0}^{L}v(t)^{2}\,\ud t}{\left(\int_{0}^{L}\int_{0}^{L} \mathcal{K}_{-\frac{\alpha}{2}}^{L}(t-\tau)\,v(\tau)^{p_{\ast}}\, v(t)^{p_{\ast}}\, \ud \tau\,\ud t\right)^{\frac{1}{p_{\ast}}}}.
	\end{equation}
	We set
	\begin{equation*}
		\mathcal{H}_{L}(v) := \int_{0}^{L}\int_{0}^{L} \mathcal{K}_{-\frac{\alpha}{2}}^{L}(t-\tau)\,v(\tau)^{p_{\ast}}\, v(t)^{p_{\ast}}\, \ud \tau\,\ud t
	\end{equation*}
	for the Hartree interaction energy. We also write $\mathcal{E}_{L}(v)$ for the numerator of $\mathscr{F}_{L}$, so that $\mathscr{F}_{L}(v) = \mathcal{E}_{L}(v) / \mathcal{H}_{L}(v)^{1/p_{\ast}}$.

	\begin{lemma}[Existence of minimizer]\label{lem:existence-minimizer}
		Let $n \geqslant 2$, $s \in (0,1)$, and $\alpha \in (\alpha_{\ast},n)$. For any $L>0$, the infimum $c(L)$ is achieved by a positive function $v_{L} \in \mathcal{C}^{\infty}(\mathbb{R})$ which solves \eqref{eq:periodic hartree 2}.
	\end{lemma}

	\begin{proof}
		Since multiplicative constants do not affect the argument, we may assume without loss of generality that $\kappa_{n,s} = 1$ and $c_{n,s}=1$.

		\medskip
		\noindent{\bf Claim~1:} {\it The infimum $c(L)$ equals $\inf_{v \in \mathcal{S}_{L}} \mathcal{E}_{L}(v)$, where $\mathcal{S}_{L} := \{ v \in H_{L}^{s} : \mathcal{H}_{L}(v) = 1 \}$, and $c(L) > 0$.}

		\noindent{\it Proof.}
		By the $0$-homogeneity of $\mathscr{F}_{L}$ ({\it i.e.}, $\mathscr{F}_{L}(\lambda v) = \mathscr{F}_{L}(v)$ for all $\lambda > 0$), the infimum $c(L)$ does not change if we restrict to the constraint set
		\begin{equation}\label{eq:constraint}
			\mathcal{S}_{L} := \left\{ v \in H_{L}^{s} : \mathcal{H}_{L}(v) = 1 \right\}.
		\end{equation}
		On $\mathcal{S}_{L}$, one has $\mathscr{F}_{L}(v) = \mathcal{E}_{L}(v)$. Since $\mathcal{E}_{L}(v) \sim \|v\|_{H_{L}^{s}}^{2}$ (equivalent norms), the minimization problem \eqref{eq:minimization} is equivalent to
		\begin{equation}\label{eq:min-constrained}
			c(L) = \inf_{v \in \mathcal{S}_{L}} \mathcal{E}_{L}(v).
		\end{equation}
		By Proposition~\ref{right hand is well-defined}, $\mathcal{H}_{L}(v) \leqslant C\|v\|_{H_{L}^{s}}^{2p_{\ast}}$, and for any nonzero $v \in H_L^s$ with $\mathcal{H}_L(v) > 0$, the rescaled function $\lambda v$ with $\lambda = \mathcal{H}_L(v)^{-1/(2p_*)}$ satisfies $\mathcal{H}_L(\lambda v) = 1$, so $\mathcal{S}_L$ is nonempty. Moreover, $c(L) > 0$ since $\mathcal{H}_{L}(v) = 1$ and the compact embedding in Proposition~\ref{compact embedding} prevent the norm from being arbitrarily small, and Claim~1 is proved.

		\medskip
		\noindent{\bf Claim~2:} {\it There exists $v_{L} \in H_{L}^{s}$ such that, up to a subsequence, $v_{j} \rightharpoonup v_{L}$ weakly in $H_{L}^{s}$ and $v_{j} \to v_{L}$ strongly in $L^{q}(0,L)$ for all admissible $q$.}

		\noindent{\it Proof.}
		Let $\{v_{j}\}_{j \in \mathbb{N}} \subset \mathcal{S}_{L}$ be a minimizing sequence, {\it i.e.}, $\mathcal{E}_{L}(v_{j}) \to c(L)$ as $j \to \infty$. Since $\mathcal{E}_{L} \sim \|\cdot\|_{H_{L}^{s}}^{2}$, the sequence $\{v_{j}\}$ is bounded in $H_{L}^{s}$. Since $H_{L}^{s}$ is a Hilbert space, by passing to a subsequence (still denoted $\{v_{j}\}$), there exists $v_{L} \in H_{L}^{s}$ such that
		\begin{equation*}
			v_{j} \rightharpoonup v_{L} \quad {\rm weakly\ in}\ H_{L}^{s}.
		\end{equation*}
		By the compact embedding $H_{L}^{s} \hookrightarrow L^{q}(0,L)$ established in Proposition~\ref{compact embedding}, it follows that
		\begin{equation}\label{eq:strong-conv}
			v_{j} \to v_{L} \quad {\rm strongly\ in}\ L^{q}(0,L)
		\end{equation}
		for all $q$ in the admissible range, {\it i.e.}, $q \in \bigl(1, \frac{2}{1-2s}\bigr)$ when $s \leqslant \tfrac{1}{2}$, and $q \geqslant 1$ when $s > \tfrac{1}{2}$, and Claim~2 is proved.

		\medskip
		\noindent{\bf Claim~3:} {\it $\mathcal{H}_{L}(v_{j}) \to \mathcal{H}_{L}(v_{L})$ as $j \to \infty$ and $\mathcal{H}_{L}(v_{L}) = 1$.}

		\noindent{\it Proof.}
		We show that $\mathcal{H}_{L}(v_{j}) \to \mathcal{H}_{L}(v_{L})$ as $j \to \infty$, and consequently $\mathcal{H}_{L}(v_{L}) = 1$.

		Indeed, this follows from the structure of the Hartree interaction energy. More precisely, we write
		\begin{equation*}
			\mathcal{H}_{L}(v_{j}) - \mathcal{H}_{L}(v_{L}) = \int_{0}^{L}\int_{0}^{L} \mathcal{K}_{-\frac{\alpha}{2}}^{L}(t-\tau) \left( v_{j}(\tau)^{p_{\ast}} v_{j}(t)^{p_{\ast}} - v_{L}(\tau)^{p_{\ast}} v_{L}(t)^{p_{\ast}} \right)\,\ud \tau\,\ud t.
		\end{equation*}
		We add and subtract $v_{j}(\tau)^{p_{\ast}} v_{L}(t)^{p_{\ast}}$ and apply the triangle inequality. For the first term, by the mean value theorem, one has
		\begin{equation*}
			\left| v_{j}(t)^{p_{\ast}} - v_{L}(t)^{p_{\ast}} \right| \leqslant p_{\ast}\left( |v_{j}(t)|^{p_{\ast}-1} + |v_{L}(t)|^{p_{\ast}-1} \right) |v_{j}(t) - v_{L}(t)|.
		\end{equation*}
		Since $\{v_{j}\}$ is bounded in $H_{L}^{s}$ and hence bounded in $L^{q}(0,L)$ for all admissible $q$, the same H\"older and Hardy--Littlewood--Sobolev estimates as in the proof of Proposition~\ref{right hand is well-defined}, together with the strong convergence \eqref{eq:strong-conv}, give
		\begin{equation*}
			\left|\mathcal{H}_{L}(v_{j}) - \mathcal{H}_{L}(v_{L})\right| \to 0 \quad {\rm as} \quad j \to \infty.
		\end{equation*}
		Here, the required exponents $q = p_{\ast}$ and $q = 2p_{\ast}/(1+\alpha)$ (the latter from the Hardy--Littlewood--Sobolev inequality when $\alpha < 1$) both lie in the admissible range \eqref{compact of q}, since $(n+\alpha)(1-2s) < 2(n-2s)$ whenever $\alpha < n$ and $n \geqslant 2$.
		In particular, $v_{L} \not\equiv 0$ since $\mathcal{H}_{L}(v_{L}) = 1$, and Claim~3 is proved.

		\medskip
		\noindent{\bf Claim~4:} {\it The infimum $c(L)$ is attained, {\it i.e.}, $\mathcal{E}_{L}(v_{L}) = c(L)$.}

		\noindent{\it Proof.}
		By the weak lower semicontinuity of $\mathcal{E}_{L}$ (each summand is convex and strongly continuous in $H_{L}^{s}$, hence weakly lower semicontinuous), one has
		\begin{equation*}
			\mathcal{E}_{L}(v_{L}) \leqslant \liminf_{j \to \infty} \mathcal{E}_{L}(v_{j}) = c(L).
		\end{equation*}
		Since $v_{L} \in \mathcal{S}_{L}$, the reverse inequality $\mathcal{E}_{L}(v_{L}) \geqslant c(L)$ is immediate. Therefore $\mathcal{E}_{L}(v_{L}) = c(L)$, and the infimum is attained.

		Since $v_{L}$ minimizes $\mathcal{E}_{L}$ subject to $\mathcal{H}_{L}(v) = 1$, there exists a Lagrange multiplier $\mu \in \mathbb{R}$ such that $D\mathcal{E}_{L}(v_{L}) = \mu\, D\mathcal{H}_{L}(v_{L})$. By the symmetry of $\mathcal{K}_{s}^{L}$ and $\mathcal{K}_{-\alpha/2}^{L}$, for all $\varphi \in H_{L}^{s}$ we have
		\begin{align*}
			D\mathcal{E}_{L}(v_{L})[\varphi] &= 2\int_{0}^{L} \mathcal{L}_{s}^{L} v_{L}(t)\, \varphi(t)\, \ud t, \\
			D\mathcal{H}_{L}(v_{L})[\varphi] &= 2p_{\ast} \int_{0}^{L} \bigl(\mathcal{K}_{-\alpha/2}^{L} \ast v_{L}^{p_{\ast}}\bigr)(t)\, v_{L}(t)^{p_{\ast}-1}\, \varphi(t)\, \ud t,
		\end{align*}
		so the Euler--Lagrange equation reads
		\begin{equation*}
			\mathcal{L}_{s}^{L} v_{L} = \mu\, p_{\ast} \bigl(\mathcal{K}_{-\alpha/2}^{L} \ast v_{L}^{p_{\ast}}\bigr) v_{L}^{p_{\ast}-1}.
		\end{equation*}
		Testing with $\varphi = v_{L}$ gives $\mu = c(L)/p_{\ast} > 0$, since $c(L) > 0$ by Claim~3. Rescaling $v_{L} \mapsto \lambda v_{L}$ with $\lambda = (\mu p_{\ast})^{1/(2(p_{\ast}-1))}$ absorbs the multiplier, and the resulting function solves \eqref{eq:periodic hartree 2}. This proves Claim~4.

		\medskip
		\noindent{\bf Claim~5:} {\it The minimizer $v_{L}$ can be chosen to be a positive function in $\mathcal{C}^{\infty}(\mathbb{R})$.}

		\noindent{\it Proof.}
        By Proposition~\ref{L infty for s<1/2} and Remark~\ref{L infty for 0<s<1}, we have $v_L\in L^\infty(\mathbb{R})$. It then follows from  Proposition~\ref{Cbeta estimate} that $v_{L} \in \mathcal{C}^{\beta}(\mathbb{R})$ for some $\beta \in (0,1)$. 
        
		We claim that $v_{L}$ can be chosen to be strictly positive. Indeed, for any $t, \tau \in (0,L)$, the pointwise inequality
		\begin{equation*}
			(v_{L}(t) - v_{L}(\tau))^{2} - (|v_{L}(t)| - |v_{L}(\tau)|)^{2} = 2\bigl(|v_{L}(t)|\,|v_{L}(\tau)| - v_{L}(t)\,v_{L}(\tau)\bigr) \geqslant 0
		\end{equation*}
		implies that $\mathcal{E}_{L}(|v_{L}|) \leqslant \mathcal{E}_{L}(v_{L})$. Since $\mathcal{H}_{L}(|v_{L}|) = \mathcal{H}_{L}(v_{L}) = 1$, one has $\mathscr{F}_{L}(|v_{L}|) \leqslant \mathscr{F}_{L}(v_{L}) = c(L)$, so $|v_{L}|$ is also a minimizer. We may therefore assume $v_{L} \geqslant 0$. The normalization $\mathcal{H}_{L}(v_{L}) = 1$ guarantees $v_{L} \not\equiv 0$, and the strong maximum principle (Proposition~\ref{maximum principle}) then gives $v_{L} > 0$.
        
         Since $v_{L} > 0$ and $v_{L} \in \mathcal{C}^{\beta}(\mathbb{R})$ for some $\beta \in (0,1)$, Remark~\ref{remark C infty} provides the bootstrap $v_{L} \in \mathcal{C}^{\infty}(\mathbb{R})$, and Claim~5 is proved.
	\end{proof}

	\subsection{The minimizer is nonconstant for large periods}
	Lemma~\ref{lem:existence-minimizer} produces a smooth periodic solution $v_{L}$ for every $L > 0$, but this minimizer could {\it a priori} be the constant solution. We now show that for $L \gg 1$, the minimizer is nonconstant, by comparing the energy of the constant profile with that of a localized test function.

	\begin{lemma}[Energy of the constant solution]\label{lem:constant-energy}
		Let $n \geqslant 2$, $s \in (0,1)$, and $\alpha \in (0,n)$. For every $L > 0$, the constant function $v \equiv \mathbf{c}_{L}$ with
		\begin{equation}\label{eq:constant-state}
			\mathbf{c}_{L} := \left( \frac{c_{n,s}}{\int_{-\infty}^{+\infty}\mathcal{K}_{-\frac{\alpha}{2}}(\tau)\,\ud\tau} \right)^{\frac{1}{2(p_{\ast}-1)}}
		\end{equation}
		solves \eqref{eq:periodic hartree 2}. (We note that $\mathbf{c}_L$ is independent of $L$.) Moreover, one has
		\begin{equation}\label{eq:energy-constant}
			\mathscr{F}_{L}(\mathbf{c}_{L}) = c_{n,s}\, L^{1-\frac{1}{p_{\ast}}} \left( \int_{-\infty}^{+\infty} \mathcal{K}_{-\frac{\alpha}{2}}(\tau)\,\ud\tau \right)^{-\frac{1}{p_{\ast}}}.
		\end{equation}
	\end{lemma}

	\begin{proof}
		For a positive constant $v\equiv c$, the Gagliardo seminorm vanishes and \eqref{eq:periodic hartree 2} reduces to
		\begin{equation*}
			c_{n,s}\,\mathbf{c} = \mathbf{c}^{2p_{\ast}-1} \int_{-\infty}^{+\infty} \mathcal{K}_{-\frac{\alpha}{2}}(\tau)\,\ud\tau,
		\end{equation*}
		which gives \eqref{eq:constant-state}. By Lemma~\ref{Pro of ker}, we know that 
        \begin{equation*}
           0 < \int_{-\infty}^{+\infty} \mathcal{K}_{-\frac{\alpha}{2}}(\tau)\,\ud\tau < \infty,
           \end{equation*}
           which implies that $0 < c_L < \infty$. A direct substitution into \eqref{eq:Rayleigh} yields \eqref{eq:energy-constant}.
	\end{proof}

	\begin{remark}\label{rem:constant-energy-grows}
		From Lemma~\ref{lem:constant-energy}, we see that as $L \to \infty$,
		\begin{equation}\label{eq:constant-energy-grows}
			\mathscr{F}_{L}(\mathbf{c}_{L}) = c_{n,s}\, L^{1-\frac{1}{p_{\ast}}} \left( \int_{-\infty}^{+\infty} \mathcal{K}_{-\frac{\alpha}{2}}(\tau)\,\ud\tau \right)^{-\frac{1}{p_{\ast}}} \to +\infty.
		\end{equation}
		In other words, the energy of the constant solution grows as $L^{1-1/p_{\ast}}$.
	\end{remark}

	\begin{proposition}[Nonconstant minimizer for large periods]\label{prop:nonconstant}
		Let $n \geqslant 2$, $s \in (0,1)$, and $\alpha \in (\alpha_{\ast},n)$. There exists $L_{0} > 0$ such that for every $L \geqslant L_{0}$, the minimizer $v_{L}$ from Lemma~\ref{lem:existence-minimizer} satisfies $v_{L} \not\equiv \mathbf{c}_{L}$.
	\end{proposition}

	\begin{proof}
		We construct an explicit test function whose energy remains bounded as $L \to \infty$, while the energy of the constant solution diverges by \eqref{eq:constant-energy-grows}. This comparison forces the minimizer to be nonconstant.

        Using the $L$-periodicity of any $v \in H^{s}_{L}$, we can rewrite $\mathscr{F}_{L}(v)$ in the following form
        \begin{align*}
            \mathscr{F}_{L}(v)=&\frac{\frac{\kappa_{n,s}}{2}\int_{0}^{L}\int_{0}^{L}\left(v(t)-v(\tau)\right)^{2}\mathcal{K}_{s}^{L}(t-\tau)\,\ud \tau\, \ud t+c_{n, s} \int_{0}^{L}v(t)^{2}\,\ud t}{\left(\int_{0}^{L}\int_{0}^{L} \mathcal{K}_{-\frac{\alpha}{2}}^{L}(t-\tau)\,v(\tau)^{p_{\ast}}\, v(t)^{p_{\ast}}\, \ud \tau\,\ud t\right)^{\frac{1}{p_{\ast}}}}\\
            =&\frac{\frac{\kappa_{n,s}}{2}\int_{-L/2}^{L/2}\int_{-\infty}^{+\infty}\left(v(x+\tau)-v(\tau)\right)^{2}\mathcal{K}_{s}(x)\,\ud x\, \ud \tau+c_{n, s} \int_{-L/2}^{L/2}v(t)^{2}\,\ud t}{\left(\int_{-L/2}^{L/2}\int_{-L/2}^{L/2} \mathcal{K}_{-\frac{\alpha}{2}}^{L}(t-\tau)\,v(\tau)^{p_{\ast}}\, v(t)^{p_{\ast}}\, \ud \tau\,\ud t\right)^{\frac{1}{p_{\ast}}}},
        \end{align*}
        where $x=t-\tau$.

        We define
        \begin{equation}
            b(t)=\left(\frac{e^{t}}{1+e^{2t}}\right)^{\frac{n-2s}{2}}
        \end{equation}
        and take a cut-off function $\eta \in \mathcal{C}_{c}^{\infty}(\mathbb{R})$ which is identically $1$ in the ball of radius $L/4$ and vanishes outside the ball of radius $L/2$. Let $w_{L}(t)=b(t)\eta(t)$, and let $\tilde{w}_{L}$ denote the $L$-periodic extension of $w_{L}$, which satisfies
		\begin{equation*}
			\tilde{w}_{L}(t) = w_{L}(t) \quad {\rm if} \quad t \in (-L/2,L/2); \quad \tilde{w}_{L}(t+L)=\tilde{w}_{L}(t).
		\end{equation*}
		Since $w_{L}$ is smooth and compactly supported in $(-L/2,L/2)$, its periodic extension $\tilde{w}_{L}$ belongs to $\mathcal{C}^{\infty}(\mathbb{R}) \cap H_{L}^{s}$, and in particular $\mathscr{F}_{L}(\tilde{w}_{L})$ is well-defined.
	    We will show that there exist $L_{0} > 0$ and $C > 0$, both independent of $L$, such that for all $L \geqslant L_{0}$,
        \begin{equation*}
            \mathscr{F}_{L}(\tilde{w}_L)\leqslant C.
        \end{equation*}

		We first estimate the numerator. For the $L^{2}$-term,
		\begin{equation}\label{numerator L2 term}
			\int_{-L/2}^{L/2} \tilde{w}_{L}(t)^{2}\,\ud t \leqslant \int_{-L/2}^{L/2} b(t)^{2}\,\ud t \leqslant \|b\|^{2}_{L^{2}(\mathbb{R})}.
		\end{equation}
		For the Gagliardo seminorm, we fix a small constant $\varepsilon > 0$, independent of $L$, such that the asymptotic behaviors of $\mathcal{K}_{s}(x)$ near zero and at infinity hold in the regions $|x| \leqslant \varepsilon$ and $|x| > \varepsilon$, respectively. We split the integral into the following two parts.
		\begin{align*}
			\int_{-L/2}^{L/2}\int_{-\infty}^{+\infty} \left(\tilde{w}_{L}(x+\tau) - \tilde{w}_{L}(\tau)\right)^{2} \mathcal{K}_{s}(x)\,\ud x\,\ud \tau 
            =:&I_{1}+I_{2},
		\end{align*}
        where
        \begin{equation*}
            I_{1}=\int_{-L/2}^{L/2}\int_{-\varepsilon}^{\varepsilon} \left(\tilde{w}_{L}(x+\tau) - \tilde{w}_{L}(\tau)\right)^{2} \mathcal{K}_{s}(x)\,\ud x\,\ud \tau 
        \end{equation*}
        and
        \begin{equation*}
            I_{2}=\int_{-L/2}^{L/2}\int_{\mathbb{R}\setminus [-\varepsilon,\varepsilon]} \left(\tilde{w}_{L}(x+\tau) - \tilde{w}_{L}(\tau)\right)^{2} \mathcal{K}_{s}(x)\,\ud x\,\ud \tau.
        \end{equation*}
        For $I_{1}$, we use the Taylor expansion of $\tilde{w}_L$. More precisely, for $|x|\leqslant \varepsilon$, we have
        \begin{equation*}
            \tilde{w}_L(\tau+x) - \tilde{w}_L(\tau)= x\,\tilde{w}_L'(\tau) + O(x^2),
        \end{equation*}
        where the remainder is uniform in $\tau$ and $L$, since $\|\tilde{w}_L''\|_{L^\infty} \leqslant \|b''\|_{L^\infty} + C\|b'\|_{L^\infty}/L + C\|b\|_{L^\infty}/L^2$, which is bounded independently of $L$. Hence, combined with Lemma~\ref{Pro of ker}, we obtain that
        \begin{align}\label{numerator I1}\nonumber
            I_1 \lesssim \int_{-L/2}^{L/2} \int_{- \varepsilon}^{\varepsilon}\frac{x^2 |\tilde{w}_L'(\tau)|^2}{|x|^{1+2s}}\, \ud x\, \ud\tau&\leqslant 2\int_{-L/2}^{L/2}\left( \eta '(\tau)^{2}b(\tau)^{2}+\eta(\tau)^{2}b'(\tau)^{2}\right)\,\ud\tau\int_{-\varepsilon}^{\varepsilon} |x|^{1-2s}\,\ud x,\\
                &\leqslant\frac{C}{1-s}\varepsilon^{2-2s}\left(\frac{1}{L^{2}}\|b\|^{2}_{L^{2}(\mathbb{R})}+ \|b'\|^{2}_{L^{2}(\mathbb{R})}\right),
        \end{align}
		where the constant $C>0$ is independent of $L$. For $I_{2}$, using the exponential decay of $\mathcal{K}_{s}$, we have
        \begin{align}\label{numerator I2}\nonumber
            I_{2}&\leqslant C \int_{-L/2}^{L/2} \int_{\mathbb{R}\setminus[-\varepsilon,\varepsilon]}\left(\tilde{w}_L(x+\tau) - \tilde{w}_L(\tau)\right)^2e^{-|x|\frac{n+2s}{2}}\, \ud x\, \ud\tau \\ \nonumber
            &\leqslant C\int_{\mathbb{R}\setminus[-\varepsilon,\varepsilon]}e^{-|x|\frac{n+2s}{2}}
            \left(\int_{-L/2}^{L/2} \tilde{w}_L(x+\tau)^2\, \ud\tau \right)\,\ud x+\left(\int_{-L/2}^{L/2} \tilde{w}_L(\tau)^2\, \ud\tau \right)\left(\int_{\mathbb{R}\setminus[-\varepsilon,\varepsilon]}e^{-|x|\frac{n+2s}{2}}\, \ud x \right)\\
            &\leqslant C e^{-\frac{n+2s}{2}\varepsilon} \|b\|_{L^{2}(\mathbb{R})}^{2},
        \end{align}
        where the constant $C>0$ is independent of $L$.

		Now we consider the denominator. For $L \geqslant 2$, the Hartree energy is bounded below by the self-interaction within a single period. Indeed, one has
		\begin{align}\label{denominator-lower}
			\mathcal{H}_{L}(\tilde{w}_{L}) &= \int_{-L/2}^{L/2}\int_{-L/2}^{L/2} \mathcal{K}_{-\frac{\alpha}{2}}^{L}(t-\tau)\, \tilde{w}_{L}(\tau)^{p_{\ast}}\, \tilde{w}_{L}(t)^{p_{\ast}}\, \ud\tau\,\ud t \nonumber \\\nonumber
			&\geqslant \int_{-1}^{1}\int_{-1}^{1} \mathcal{K}_{-\frac{\alpha}{2}}(t-\tau)\, b(\tau)^{p_{\ast}}\, b(t)^{p_{\ast}}\, \ud\tau\,\ud t  \\
            & \geqslant b(1)^{2p_{\ast}} \int_{-1}^{1}\int_{-1}^{1}\mathcal{K}_{-\frac{\alpha}{2}}(t-\tau)\,\ud \tau\,\ud t.
		\end{align}
        
        Combining \eqref{numerator L2 term}, \eqref{numerator I1}, \eqref{numerator I2} and \eqref{denominator-lower}, we obtain
		\begin{equation} \label{FL tildewL}
			\mathscr{F}_{L}(\tilde{w}_{L}) = \frac{\mathcal{E}_{L}(\tilde{w}_{L})}{\mathcal{H}_{L}(\tilde{w}_{L})^{1/p_{\ast}}} \leqslant \frac{C\left(\|b\|^{2}_{L^{2}(\mathbb{R})}+ \|b'\|^{2}_{L^{2}(\mathbb{R})}
            \right)}{|b(1)|^{2} \left(\int_{-1}^{1}\int_{-1}^{1}\mathcal{K}_{-\frac{\alpha}{2}}(t-\tau)\,\ud \tau\,\ud t \right)^{1/p_{\ast}}},
		\end{equation}
        where the constant $C$ depends only on $n,s$ and $\varepsilon$ (and is independent of $L$ for $L \geqslant 2$). Since $b(t) = (2\cosh t)^{-(n-2s)/2}$ decays as $e^{-(n-2s)|t|/2}$ with $n > 2s$, both $b$ and $b'$ belong to $L^{2}(\mathbb{R})$.

		\medskip
		We now conclude the proof. Since the right-hand side of \eqref{FL tildewL} is independent of $L$ while $\mathscr{F}_{L}(\mathbf{c}_{L}) \to \infty$ by \eqref{eq:constant-energy-grows}, there exists $L_{0} > 0$ such that for all $L \geqslant L_{0}$ one has
		\begin{align*}
			c(L) \leqslant \mathscr{F}_{L}(\tilde{w}_{L}) 
            \leqslant \frac{C\left(\|b\|^{2}_{L^{2}(\mathbb{R})}+ \|b'\|^{2}_{L^{2}(\mathbb{R})}
            \right)}{|b(1)|^{2} \left(\int_{-1}^{1}\int_{-1}^{1}\mathcal{K}_{-\frac{\alpha}{2}}(t-\tau)\,\ud \tau\,\ud t \right)^{1/p_{\ast}}}&< c_{n,s}\, L^{1-\frac{1}{p_{\ast}}} \left( \int_{-\infty}^{+\infty} \mathcal{K}_{-\frac{\alpha}{2}}(\tau)\,\ud\tau \right)^{-\frac{1}{p_{\ast}}}\\
            &= \mathscr{F}_{L}(\mathbf{c}_{L}).
		\end{align*}
		Therefore $v_{L} \not\equiv \mathbf{c}_{L}$, since the minimizer attains $c(L) < \mathscr{F}_{L}(\mathbf{c}_{L})$.
	\end{proof}

	\begin{remark}[Role of large periods]\label{rem:large-period}
		The energy $\mathscr{F}_{L}(w_{L})$ of the compactly supported test function stabilizes at $O(1)$ as $L \to \infty$, since its profile does not change and cross-period interactions decay exponentially. In contrast, $\mathscr{F}_{L}(\mathbf{c}_{L}) \sim L^{1-1/p_{\ast}} \to \infty$, so the localized bump beats the constant for $L \gg 1$. This is the Hartree analog of the energy comparison in~\cite{MR3694655}.
	\end{remark}

	\subsection{Proof of the main theorem}
	Finally, we prove the existence result.

	\begin{proof}[Proof of Theorem~\ref{thm:main}]
		Let $L_{0} > 0$ be as in Proposition~\ref{prop:nonconstant}, and fix $T := L \geqslant L_{0}$. By Lemma~\ref{lem:existence-minimizer}, there exists a positive, smooth, $L$-periodic solution $v_{L} \in \mathcal{C}^{\infty}(\mathbb{R})$ to \eqref{eq:periodic hartree 2}. By Proposition~\ref{prop:nonconstant}, $v_{L}$ is nonconstant.

		Let us set
		\begin{equation*}
			u_{L}(x) := |x|^{-\frac{n-2s}{2}} v_{L}(\ln|x|), \qquad x \in \mathbb{R}^{n}\setminus\{0\}.
		\end{equation*}
		By the equivalence between \eqref{eq:frac Hartree} and \eqref{eq:periodic hartree 2} established in \S\ref{sec:prelim}, $u_{L}$ is a smooth positive solution to \eqref{eq:frac Hartree} on $\mathbb{R}^{n}\setminus\{0\}$. Since $v_{L}$ is $L$-periodic and bounded between two positive constants, one has
		\begin{equation*}
			0 < c_{1} \leqslant |x|^{\frac{n-2s}{2}} u_{L}(x) \leqslant c_{2} \quad {\rm for\ all} \quad x \in \mathbb{R}^{n}\setminus\{0\}
		\end{equation*}
		for constants $c_{1}, c_{2} > 0$ depending on $L$. Moreover, the $L$-periodicity of $v_{L}$ gives the invariance 
        \[
        u_{L}(e^{L}x) = e^{-\frac{(n-2s)L}{2}} u_{L}(x).
        \]
        Finally, since $v_{L}$ is nonconstant, $u_{L}$ is indeed a Delaunay solution. This completes the proof.
	\end{proof}

    \begin{remark}
We emphasize that the Delaunay-type solutions obtained in Theorem~\ref{thm:main} satisfy all the assumptions of Theorem~\ref{symmetric}. Indeed, they are positive classical solutions of \eqref{eq:frac Hartree} in
$\mathbb{R}^n\setminus\{0\}$, belong to
$L_s(\mathbb{R}^n)\cap \mathcal{C}^{1,1}_{\loc}(\mathbb{R}^n\setminus\{0\}) \cap L^{1}_{\loc}(\mathbb{R}^{n})$,
and satisfy
\begin{equation*}
    0<c_1\leqslant |x|^{\frac{n-2s}{2}}u(x)\leqslant c_2
\end{equation*}
for some constants $c_1,c_2>0$.
\end{remark}


\section*{Acknowledgments}
This work was partially supported by Funda\c c\~ao de Amparo \`a Pesquisa do Estado de S\~ao Paulo (FAPESP), Conselho Nacional de Desenvolvimento Cient\'ifico e Tecnol\'ogico (CNPq), National Science Foundation of China (NSFC), and Natural Science Foundation of Zhejiang Province (ZJNSF). 
J.H.A. was supported by FAPESP \#2021/15567-8, and CNPq \#409764/2023-0, \#443594/2023-6, \#441922/2023-6, and \#306014/2025-4.
P.P. was supported by FAPESP \#2016/23746-6 and CNPq \#313773/2021-1. 
M.Y.\ were partially supported by NSFC \#12471114 and ZJNSF \#LZ26A010002. 



\begin{thebibliography}{10}

		\bibitem{AndradeFengPiccioneYang2025}
		J.~H. Andrade, T.~Feng, P.~Piccione and M.~Yang,
		Local asymptotics and Pohozaev identities for singular solutions of critical Hartree equations,
		preprint (2025), \href{https://arxiv.org/abs/2505.19021}{arXiv:2505.19021}.

        \bibitem{CabreSire2014}
        X. Cabr\'e and Y. Sire, Nonlinear equations for fractional Laplacians I: Regularity, maximum principles, and Hamiltonian estimates,
        \emph{Ann. Inst. H. Poincar\'e Anal. Non Lin\'eaire} \textbf{31} (2014), 23--53.

        \bibitem{Caffarelli-Gidas-Spruck}
		L.~Caffarelli, B.~Gidas, and J.~Spruck, Asymptotic symmetry and local behavior of semilinear elliptic equations with critical {S}obolev growth, Comm. Pure Appl. Math., {\bf 42} (1989), 271--297.

        \bibitem{CaffarelliJinSireXiong2014}
        L.~Caffarelli, T.~Jin, Y.~Sire, and J.~Xiong, Local analysis of solutions of fractional semi-linear elliptic equations with isolated singularities,
        \emph{Arch. Ration. Mech. Anal.}
        \textbf{213} (2014), no.~1, 245--268.

		\bibitem{CaffarelliSilvestre2007}
		L.~Caffarelli and L.~Silvestre,
		An extension problem related to the fractional Laplacian,
		\emph{Comm. Partial Differential Equations}
		\textbf{32} (2007), no.~7--9, 1245--1260.

		\bibitem{CaseChang16}
		J.~S. Case and S.-Y.~A. Chang,
		On fractional GJMS operators,
		\emph{Commun. Pure Appl. Math.}
		\textbf{69} (2016), no.~6, 1017--1061.

		\bibitem{CG11}
		S.-Y.~A. Chang and M.~d.~M. Gonz\'alez,
		Fractional Laplacian in conformal geometry,
		\emph{Adv. Math.}
		\textbf{226} (2011), no.~2, 1410--1432.

        \bibitem{Chang-Yang} 
		S. Y. A. Chang, and P. C. Yang, On uniqueness of solutions of n-th order differential equations in conformal geometry, Math. Res. Lett., {\bf 4} (1997), 91--102.

        \bibitem{Chen-Li1}  
		W.~Chen, and C.~Li, Classification of solutions of some nonlinear elliptic equations, Duke Math. J., {\bf 63} (1991), 615--622.
		
		\bibitem{Chen-Li2} 
		W. Chen, and C. Li, A necessary and sufficient condition for the Nirenberg problem, Comm. Pure Appl. Math., {\bf 48} (1995),  no. 6, 657--667.

        \bibitem{MR2200258}
		W.~Chen, C.~Li and B.~Ou, Classification of solutions for an integral equation,
		\emph{Comm. Pure Appl. Math.} {\bf 59} (2006) 330--343.

        \bibitem{Chou-Wan}
		K. S. Chou, and Y. H. Wan, Asymptotic radial symmetry for solutions of $-\Delta u + e^u = 0$ in a punctured disc, Pacific J. Math., {\bf 163} (1994), no. 2, 269--276.

        \bibitem{ChenLiOu2005}
        W.~Chen, C.~Li, and B.~Ou,
        Qualitative properties of solutions for an integral equation,
        \emph{Discrete Contin. Dyn. Syst.}
        \textbf{12} (2005), no.~2, 347--354.

        \bibitem{ChenLiMa}
        W.~Chen, Y.~Li and P.~Ma, \emph{The Fractional Laplacian}, World Scientific Publishing Co. Pte. Ltd., Singapore, 2019.

        \bibitem{DaiFangQin2018}
        W.~Dai, Y.~Fang and G.~Qin, Classification of positive solutions to fractional order {H}artree equations via a direct method of moving planes, \emph{J. Differential Equations} {\bf 265} (2018), 2044--2063.

        \bibitem{DaiHuangQinWangFang2019}
        W.~Dai, J.~Huang, Y.~Qin, B.~Wang and Y.~Fang, Regularity and classification of solutions to static {H}artree equations involving fractional {L}aplacians, \emph{Discrete Contin. Dyn. Syst.} {\bf 39} (2019), 1389--1403.


		\bibitem{DelaTorreGonzalez18}
		A.~DelaTorre and M.~del~Mar Gonz\'{a}lez,
		Isolated singularities for a semilinear equation
		for the fractional Laplacian arising in conformal geometry,
		\emph{Rev. Mat. Iberoam.}
		{\bf 34} (2018), 1645--1678.

		\bibitem{MR3694655}
		A.~DelaTorre, M.~del Pino, M.~d.~M. Gonz\'{a}lez and J.~Wei, Delaunay-type
		singular solutions for the fractional {Y}amabe problem, \emph{Math. Ann.}
		{\bf 369} (2017), 597--626.

		\bibitem{DiCastroKuusiPalatucci2014}
		A.~Di Castro, T.~Kuusi and G.~Palatucci,
		Nonlocal Harnack inequalities,
		\emph{J. Funct. Anal.} {\bf 267} (2014), no.~6, 1807--1836.

		\bibitem{DiNezza2012}
		E.~Di Nezza, G.~Palatucci and E.~Valdinoci,
		Hitchhiker's guide to the fractional Sobolev spaces,
		\emph{Bull. Sci. Math.} {\bf 136} (2012), no.~5, 521--573.

		\bibitem{DongKim2013}
		H.~Dong and D.~Kim,
		Schauder estimates for a class of non-local elliptic equations,
		\emph{Discrete Contin. Dyn. Syst.} {\bf 33} (2013), 2319--2347.

        \bibitem{MR4027015}
		L.~Du and M.~Yang, Uniqueness and nondegeneracy of solutions for a critical
		nonlocal equation, \emph{Discrete Contin. Dyn. Syst.} {\bf 39} (2019)
		5847--5866.

		\bibitem{FabesKenigSerapioni1982}
		E.~B. Fabes, C.~E. Kenig and R.~P. Serapioni,
		The local regularity of solutions of degenerate elliptic equations,
		\emph{Comm. Partial Differential Equations} {\bf 7} (1982), no.~1, 77--116.

        \bibitem{Feng-Yang-Zhou}
        T. Feng, M. Yang, and X. Zhou,
        Asymptotic behavior of solutions to a planar Hartree equation with isolated singularities, preprint, arXiv:2602.03559.

        \bibitem{FrankKonig2019}
		R.~L. Frank and T.~K\"{o}nig, Classification of positive singular solutions to
		a nonlinear biharmonic equation with critical exponent, \emph{Anal. PDE} {\bf 12} (2019) 1101--1113.

		\bibitem{FrankLenzmannSilvestre2016}
		R.~L. Frank, E.~Lenzmann and L.~Silvestre,
		Uniqueness of radial solutions for the fractional Laplacian,
		\emph{Comm. Pure Appl. Math.} {\bf 69} (2016), no.~9, 1671--1726.

        \bibitem{MR3817173}
		F.~Gao and M.~Yang, The {B}rezis-{N}irenberg type critical problem for the
		nonlinear {C}hoquard equation, \emph{Sci. China Math.} {\bf 61} (2018)
		1219--1242.

        \bibitem{Giusti2003}
        E.~Giusti, \emph{Direct Methods in the Calculus of Variations},
        World Scientific, River Edge, NJ, 2003.

        \bibitem{Gluck}
		M. Gluck, Classification of solutions to an elliptic equation on $\mathbb{R}^{2}$ with nonlocal nonlinearity, Discrete and
		Continuous Dynamical Systems, {\bf 45} (2025), 4262--4299.

		\bibitem{MR3148060}
		M.~d.~M. Gonz\'{a}lez and J.~Qing,
		Fractional conformal Laplacians and fractional Yamabe problems,
		\emph{Anal.\ PDE} \textbf{6} (2013), 1535--1576.

        \bibitem{MR3978520}
		L.~Guo, T.~Hu, S.~Peng and W.~Shuai, Existence and uniqueness of solutions for
		{C}hoquard equation involving {H}ardy-{L}ittlewood-{S}obolev critical
		exponent, \emph{Calc. Var. Partial Differential Equations} {\bf 58} (2019)
		Paper No. 128, 34.

        \bibitem{Guo-Peng}
		Y. Guo, and S. Peng, Asymptotic behavior and classification of solutions to Hartree type equations with exponential nonlinearity, J. Geom. Anal., {\bf 34} (2024), no. 1, Paper No. 23, 21 pp.

        \bibitem{GuoHuangWangWei2020}
		Z.~Guo, X.~Huang, L.~Wang and J.~Wei, On Delaunay solutions of a biharmonic
		elliptic equation with critical exponent, \emph{J. Anal. Math.} {\bf 140}
		(2020) 371--394.

        \bibitem{Guo-Liu}
		Z. Guo, and Z. Liu, Asymptotic behavior of solutions for some elliptic equations in exterior Domains, Pacific J. Math., {\bf 309} (2020), no. 2, 333--352.

		\bibitem{MR1544927}
		G.~H. Hardy and J.~E. Littlewood, Some properties of fractional integrals. {I},
		\emph{Math. Z.} {\bf 27} (1928), 565--606.

        \bibitem{JinLiXiong2014}
        T. Jin, Y.Y. Li and J. Xiong, On a fractional Nirenberg problem, part~I: blow up analysis and compactness of solutions,
        \emph{J. Eur. Math. Soc. (JEMS)} {\bf 16} (2014), 1111--1171.

		\bibitem{JinXiong2020}
		T.~Jin and J.~Xiong, Asymptotic symmetry and local behavior of solutions of
		higher order conformally invariant equations with isolated singularities,
		\emph{Ann. Inst. H. Poincar\'{e} C, Anal. Non Lin\'{e}aire} {\bf 38} (2021), no.~4, 1167--1216.

        

		\bibitem{Kassmann2009}
		M.~Kassmann,
		A priori estimates for integro-differential operators with measurable kernels,
		\emph{Calc. Var. Partial Differential Equations} {\bf 34} (2009), 1--21.



		\bibitem{MR717827}
		E.~H. Lieb, Sharp constants in the {H}ardy--{L}ittlewood--{S}obolev and related
		inequalities, \emph{Ann. of Math. (2)} {\bf 118} (1983), 349--374.

        \bibitem{Li2004}
        Y.~Y. Li, Remark on some conformally invariant integral equations: the method of moving spheres, \emph{J. Eur. Math. Soc.} {\bf 6} (2004), 153--180.

        \bibitem{MR1611691}
		C.-S. Lin, A classification of solutions of a conformally invariant fourth
		order equation in {${\bf R}^n$}, \emph{Comment. Math. Helv.} {\bf 73} (1998)
		206--231.

		\bibitem{Lions1985a}
		P.-L. Lions,
		The concentration-compactness principle in the calculus of variations. The locally compact case, part~1,
		\emph{Ann. Inst. H. Poincar\'{e} C, Anal. Non Lin\'{e}aire} {\bf 1} (1984), no.~2, 109--145.

		\bibitem{Lions1985b}
		P.-L. Lions,
		The concentration-compactness principle in the calculus of variations. The locally compact case, part~2,
		\emph{Ann. Inst. H. Poincar\'{e} C, Anal. Non Lin\'{e}aire} {\bf 1} (1984), no.~4, 223--283.

        \bibitem{Niu}
		Y. Niu, Classification of solutions of higher order critical Choquard equation, Commun. Pure Appl. Anal., {\bf24} (2025), no. 5, 812--839.

		\bibitem{Silvestre2006}
		L.~Silvestre,
		H\"older estimates for solutions of integro-differential equations like the fractional Laplace,
		\emph{Indiana Univ. Math. J.} {\bf 55} (2006), 1155--1174.

        \bibitem{Silvestre2007}
        L.~Silvestre,
        Regularity of the obstacle problem for a fractional power of the Laplace operator,
        \emph{Comm. Pure Appl. Math.} {\bf 60} (2007), 67--112.

		\bibitem{MR165337}
		S.~L. Sobolev, \emph{Applications of functional analysis in mathematical
			physics}, Translations of Mathematical Monographs, Vol.~7, American
		Mathematical Society, Providence, RI (1963), translated from the Russian by
		F.~E. Browder.


        \bibitem{WeiXu1999}
        J.~Wei and X.~Xu,
        Classification of solutions of higher order conformally invariant equation,
        \emph{Math. Ann.} {\bf 313} (1999), 207--228.

        \bibitem{Yang-Yang}
        H. Yang and R. Yang, On isolated singularities of the conformal Gaussian curvature equation and $Q$-curvature equation, Math. Ann., {\bf 394} (2026), Paper No. 97.


	\end{thebibliography}
\end{document}